\documentclass[journal,twoside,web]{ieeecolor}
\usepackage{generic}
\usepackage{cite}
\usepackage{amsmath,amssymb,amsfonts}
\usepackage{mathtools}
\usepackage{graphicx}
\usepackage{algorithm,algorithmic}
\usepackage{hyperref}
\usepackage{subcaption}
\usepackage{textcomp}
\usepackage{booktabs}
\usepackage{multirow}
\usepackage{siunitx}
\usepackage{xparse}
\usepackage{etoolbox}

\let\labelindent\relax
\usepackage{enumitem}

\makeatletter
\newcommand*{\saveparentlabel}{\let\parentlabel\@currentlabel}
\makeatother

\newlist{asslist}{enumerate}{1}
\AddToHook{env/asslist/begin}{\saveparentlabel}
\setlist[asslist]{
    label=\arabic*.,
    ref=\arabic*,
    labelindent=1em,
    leftmargin=*
}
\makeatletter
\def\@IEEEsectpunct{\ \,}  
\makeatother
\makeatletter
\def\@opargbegintheorem#1#2#3{\@IEEEtmpitemindent\itemindent\topsep 0pt\rmfamily \trivlist%
    \item[\hskip\labelsep{\indent\itshape #1\ #2\ #3:}]\itemindent\@IEEEtmpitemindent}
\makeatother

\makeatletter
\let\@origsubsubsection\subsubsection
\newcommand\@nonumsubsubsection[2][]{\@origsubsubsection*{#2}}
\renewcommand\subsubsection{\@ifstar\@nonumsubsubsection\@nonumsubsubsection}
\makeatother

\makeatletter
\newcommand\@plainparagraph[2][]{%
  \par\textbf{#2}\@IEEEsectpunct\ignorespaces}
\renewcommand\paragraph{\@ifstar\@plainparagraph\@plainparagraph}
\makeatother

\def\BibTeX{{\rm B\kern-.05em{\sc i\kern-.025em b}\kern-.08em
    T\kern-.1667em\lower.7ex\hbox{E}\kern-.125emX}}
\newenvironment{problem}{\textbf{Problem:}}{}

\newtheorem{remarkEnv}{Remark}
\newenvironment{remark}[1][]{%
  \ifblank{#1}{\begin{remarkEnv}}{\begin{remarkEnv}[#1]}%
}{\hfill$\blacklozenge$\end{remarkEnv}}

\newtheorem{theoremEnv}{Theorem}
\newenvironment{theorem}[1][]{%
  \ifblank{#1}{\begin{theoremEnv}}{\begin{theoremEnv}[#1]}%
}{\hfill$\square$\end{theoremEnv}}

\newtheorem{lemmaEnv}{Lemma}
\newenvironment{lemma}[1][]{%
  \ifblank{#1}{\begin{lemmaEnv}}{\begin{lemmaEnv}[#1]}%
}{\hfill$\square$\end{lemmaEnv}}

\newtheorem{corollaryEnv}{Corollary}
\newenvironment{corollary}[1][]{%
  \ifblank{#1}{\begin{corollaryEnv}}{\begin{corollaryEnv}[#1]}%
}{\hfill$\square$\end{corollaryEnv}}

\newtheorem{propositionEnv}{Proposition}
\newenvironment{proposition}[1][]{%
  \ifblank{#1}{\begin{propositionEnv}}{\begin{propositionEnv}[#1]}%
}{\hfill$\square$\end{propositionEnv}}

\newtheorem{definitionEnv}{Definition}
\newenvironment{definition}[1][]{%
  \ifblank{#1}{\begin{definitionEnv}}{\begin{definitionEnv}[#1]}%
}{\hfill$\bullet$\end{definitionEnv}}

\newtheorem{assumptionEnv}{Assumption}
\newenvironment{assumption}[1][]{%
  \ifblank{#1}{\begin{assumptionEnv}}{\begin{assumptionEnv}[#1]}%
}{\hfill$\bullet$\end{assumptionEnv}}

\DeclareMathOperator*{\argmin}{\arg\!\min}

\newcommand{\relu}{\text{ReLU}}
\newcommand{\fnom}{f_{0}}

\newcommand{\fhat}{\hat{f}}
\newcommand{\fnomapprox}{\hat{f}_{0}}
\newcommand{\gapprox}{\hat{g}}
\newcommand{\psiapprox}{\hat{\psi}}

\newcommand{\Vapprox}{\widehat{V}}
\newcommand{\gVapprox}{\nabla \widehat{V}}
\newcommand{\npd}{\omega}
\newcommand{\npdapprox}{\hat{\omega}}
\newcommand{\gV}{\nabla V}

\newcommand{\R}{\mathbb{R}}
\newcommand{\Rn}{\mathbb{R}^{n}}

\newcommand{\op}{\mathsf{P}}
\newcommand{\zpo}{\op_{\mathrm{Zu}}}

\newcommand{\triple}{(\mathcal{F}, \mathcal{V}, \mathcal{P})}
\newcommand{\triplehat}{(\widehat{\mathcal{F}}, \widehat{\mathcal{V}}, \widehat{\mathcal{P}})}
\NewDocumentCommand{\lip}{o m}{%
  \mathcal{C}^{\IfValueTF{#1}{#1,}{}\mathrm{Lip}}(#2)%
}

\NewDocumentCommand{\llip}{o m}{%
  \mathcal{C}^{\IfValueTF{#1}{#1,}{}\mathrm{Lip}}_{\mathrm{loc}}(#2)%
}

\NewDocumentCommand{\ck}{ O{} g }{%
  \mathcal{C}%
  \ifx\\#1\\\else^{#1}\fi
  \IfNoValueTF{#2}{}{(#2)}%
}
\newcommand{\cd}[1]{\mathcal{C}^1(#1)}

\newcommand{\cn}[1]{\mathcal{C}(#1)}

\newcommand{\roa}{\mathcal{R}}

\newcommand{\roaapprox}{\widehat{\roa}}

\begin{document}
\title{Learning Locally Stable Neural ODEs\\ and Their Regions of Attraction}
\author{Alice Harting, 
Karl Henrik Johansson, 
Sophie Tarbouriech, 
and Matthieu Barreau
\thanks{
This work was partially supported by the Wallenberg AI, Autonomous Systems and Software Program (WASP) funded by the Knut and Alice Wallenberg Foundation. The computations were enabled by resources provided by the National Academic Infrastructure for Supercomputing in Sweden (NAISS), partially funded by the Swedish Research Council through grant agreement no. 2022-06725. \emph{(Corresponding author: A. Harting.)}}
\thanks{
A. Harting, K. H. Johansson, and M. Barreau are with the Department of Decision and Control Systems, KTH Royal Institute of Technology, Stockholm, Sweden. They are also affiliated with Digital Futures, Stockholm, Sweden. Emails: \{\texttt{aharting}, \texttt{kallej}, \texttt{barreau}\}\texttt{@kth.se}. S. Tarbouriech is with LAAS-CNRS, Universit\'{e} de Toulouse, CNRS, Toulouse, France. Email: \texttt{sophie.tarbouriech@laas.fr}.
}}
\maketitle
\begin{abstract} We propose a class of neural ODEs that are hard-constrained to have a locally exponentially stable equilibrium point at the origin, with an exactly characterized region of attraction. The model obtains its dynamics by passing an arbitrary neural vector field through a projection layer defined by a jointly learned Lyapunov function. By design, this function satisfies a variant of the Zubov equation for the projected dynamics and is therefore a maximal Lyapunov function for the projected system. As a result, the projected dynamics are locally exponentially stable at the origin, and the region of attraction is exactly given by the 1-sublevel set of the Lyapunov function. We prove that this model class universally approximates any sufficiently smooth dynamics that are locally exponentially stable at the origin on their region of attraction, and propose an implementation that is consistent with the theory. The approach is validated experimentally on nonlinear systems with nonconvex regions of attraction. 
\end{abstract}


\section{Introduction}
\label{sec:introduction}

\IEEEPARstart{M}{any} natural and engineered systems can be modeled by ordinary differential equations (ODEs) \cite{aastrom2021feedback},
\begin{equation}\label{eq:sys}
    \frac{d}{dt}{x}(t)=f(x(t)),\quad x(0)=x_0\in D,
\end{equation} where $f:D\to\Rn$ is a locally Lipschitz function with domain $D\subset\Rn$.
A central problem in system identification is to estimate the dynamics $f$ from measurements. Classical methods include spectral analysis \cite{pintelon2012system}, prediction-error methods \cite{ljung1999sysidbook}, and subspace identification \cite{overschee1996subspace}, among others. More recently, learning-based methods \cite{pillonetto_deep_2025} have been used to leverage highly expressive models to fit nonlinear dynamics directly to observed trajectories, either as discrete-time recurrences in recurrent neural networks (RNNs), or continuous-time neural ODEs \cite{chen2018neural, rackauckas2021universaldifferentialequationsscientific}. A central challenge in learning a dynamical system is transferring properties of the original system, such as its equilibrium points and their stability, to the learned system. Indeed, when learning a neural model using data from a system with a stable equilibrium, there is no guarantee that the neural system inherits stability \cite{drgona_safe, manchester_neural_2026}. 

 



Stability can be analyzed with \emph{Lyapunov functions}  \cite{lyapunov1992general}, \cite[Chap~4]{khalil2002nonlinear}. These are positive definite, energy-like functions of the state that decrease along trajectories of the system and vanish in the equilibrium point. If there exists a Lyapunov function with strict decay near the origin, then trajectories starting sufficiently close converge to the equilibrium, which is then called \emph{asymptotically stable}, or \emph{exponentially stable} if the convergence is
exponentially fast. The set of initial states that converge to the equilibrium point forms the region of attraction (ROA), and a sublevel set of the Lyapunov function can provide an inner estimate. In the seminal work \cite{zubov1961methods}, and later refined in \cite{vannelli_maximal_1985}, it is shown that the ROA is exactly characterized by the $1$-sublevel set of a solution to a partial differential equation (PDE) parameterized by the ODE dynamics known as the Zubov equation. However, finding a Lyapunov function, let alone one that exactly identifies the ROA, is difficult in general.  

\subsection{Related Works}
For neural network dynamics, whose activation functions are often nonpolynomial, solving the Lyapunov equation \cite[Chap~8]{khalil2002nonlinear} certifies local stability for the linearization without directly providing an estimate of the ROA. The sum-of-squares (SOS) approach \cite{sospapa, helponsos} requires a polynomial reformulation of the activation functions, which increases the problem size and may introduce conservatism.
Instead, stability can be verified by analyzing the structure of the trained network. Single-layer RNNs with Rectified Linear Unit (ReLU) activations can be cast as linear threshold discrete-time systems, where conditions for global stability with quadratic Lyapunov functions reduce to a set of linear matrix inequalities (LMIs) \cite{shen_stability}.
For deep networks, the nonlinear activations can be separated from the affine part using a linear fractional transformation \cite{robust_control_zhou_doyle_glover}, 
and quadratic constraints on the activations \cite{fazylab_2022} give corresponding LMIs \cite{ebihara_l2_2021, richardson_nn_stability, noori_rnn_relu_qcs, noori_rnn_general_relu_qcs}. 
This has been applied to local stability verification with ROA estimation \cite{yin_stability, wu_stability_2022, zanette_regional_nodate}, and constrained training of closed-loop dynamics \cite{Gu2021RecurrentNN}. However, these LMIs become computationally expensive to solve for networks with a large number of neurons, and the resulting semidefinite programs are prone to numerical issues \cite{roux_validating_2018}. Moreover, ROA estimates with existing methods are generally conservative \cite{liu_wang_review_rnn_stability}. Whereas LMI-based methods search for a Lyapunov function within a class fixed beforehand (such as, quadratic, piecewise quadratic on a given partition, polynomial, or quadratic in an augmented vector involving the nonlinearities),
 the method in \cite{samanipour_relu_lyapunov} instead uses single-layer neural networks with ReLU activations, which are piecewise affine functions whose partition depends on the number of neurons. However, this approach is limited to single-layer RNN dynamics and becomes intractable for a large number of neurons.

Recently, deep neural network parameterizations of the Lyapunov function, optimized via gradient descent, have shown promising results. In \cite{richards_lyapunov_2018}, the Lyapunov function is parameterized by a positive definite, locally Lipschitz neural network and trained using a classification loss on trajectories labeled by ROA inclusion. The methods \cite{ddzubov, zubov-koopman} target a closed-form representation of a solution to the Zubov equation with supervised learning, while the methods \cite{liu_physics-informed_2025, barreau_supervised_2025} target the PDE residual with unsupervised learning using physics-informed neural networks \cite{karniadakis2021physics}. In \cite{binny_safe_2026, sayak_lyapunovdpc}, the dynamics are jointly learned with a Lyapunov function where the stability criterion similarly shapes the loss function. While the expressive parameterization is conducive to maximal estimates of the ROA, these methods enforce compliance with the stability criterion on at most a finite set of points in the domain. Hence, they either resort to probabilistic guarantees or require additional verification with satisfiability modulo theories solvers \cite{ahmed_automated_2020,abate_nn_lyapunov,fossil,zhou_neural_2022}, which certify only a more conservative sublevel set.

Instead of relying on post-hoc analysis of the trained dynamics, a different approach is to impose hard constraints on the parameterization, ensuring that the resulting system is stable by design. Early examples include energy-based methods derived from Hamiltonian \cite{greydanus2019hamiltonian, zhong2019symplectic} or Lagrangian systems \cite{cranmer2020lagrangian}, parameterizing the dynamics according to the Hamilton or Euler-Lagrange equations, respectively, which result in conservative systems by design \cite[Chap~2,~8]{classical_mechanics}. This has been generalized to port-Hamiltonian formulations giving dissipative and passive systems \cite{zhong_dissipative_2020,pmlr-v211-neary23a,cheng_learning_2024,roth_stable_2025}, which are globally asymptotically stable \cite{cheng_learning_2024,roth_stable_2025}. Another approach is to learn a Koopman representation \cite{lusch_deep_2018} and constrain the spectrum of the linear dynamics in the lifted space \cite{koopmanstable}, or to constrain the spectrum of the network weights directly \cite{tuor2020constrainedneuralordinarydifferential}. 

While the above methods achieve stability by encoding energy-based or stable linear structures, another line of work enforces Lyapunov conditions directly by projecting the neural dynamics onto the solution set that satisfies them. The foundational work \cite{manek_learning_2020} proposes a differentiable projection layer defined by the jointly learned Lyapunov function
\begin{equation}\label{eq:mk}
    \fhat\triangleq \fnomapprox-\nabla \Vapprox \frac{\relu(\nabla \Vapprox^\top \fnomapprox+\alpha \Vapprox)}{\|\nabla \Vapprox\|^2},
\end{equation}where $\alpha$ is a positive parameter, $\fnomapprox:\Rn\to\Rn$ is a standard neural network, and $\Vapprox:\Rn\to\R$ is a Lyapunov function modeled by a positive definite input-convex neural network~\cite{amos2017input}. This gives:\begin{equation*}
     \forall x\in \Rn,\quad \nabla \Vapprox(x)^\top\fhat(x) \leq -\alpha \Vapprox(x),
\end{equation*} rendering by design the origin a globally asymptotically stable equilibrium point of the ODE defined by \eqref{eq:mk}. The approach has been extended to asymptotically stable invariant sets \cite{takeishi_learning_2020}, input-to-state stability \cite{yang_input--state_2022}, and input--output stability \cite{okamoto_learning_2023}. 

However, these architectures aim for global convergence, which limits their ability to capture locally stable equilibria with bounded ROAs. We observe that if $\fnomapprox$ approximates $f$ well, then $\nabla\Vapprox^\top\fnomapprox\approx 0$ on the ROA boundary, since this set is formed by trajectories \cite[Lemma~8.1]{khalil2002nonlinear}. The error of the projected model~\eqref{eq:mk} satisfies
\begin{equation*}
    \|f-\fhat\|\leq \|f-\fnomapprox\|+\relu(\gVapprox^\top \fnomapprox+\alpha \Vapprox)/\|\gVapprox\|,
\end{equation*}
and on the boundary the second term becomes approximately $\alpha\Vapprox/\|\gVapprox\|$. Hence, this is small only if $\Vapprox/\|\gVapprox\|$ is small; with $\Vapprox$ positive definite,
this forces $\gVapprox$ to blow up. Methods in \cite{sochopoulos_learning_2024, okamoto_learning_2023, massaroli_stable_2020} instead use positive semidefinite Lyapunov functions with multiple minima while retaining the global convergence guarantee, thereby admitting multiple locally asymptotically stable equilibria. However, these methods lack an explicit ROA characterization for a given equilibrium point.

Moreover, universal approximation for these projection layers remains largely unaddressed. 
Recently, the authors of \cite{min_hardnet_2025} generalized the design \eqref{eq:mk} so that the projected dynamics $\fhat$ satisfies input-dependent constraints of a general form, 
\begin{equation}\label{eq:hardnet}
    \forall x\in \Rn, \quad b_l(x) \leq A(x)f(x)\leq b_u(x),
\end{equation} with $A(x)$ continuous and of full row rank at every $x$, and proved universal approximation over the class of functions satisfying them. 
This was generalized in \cite{zhao2026caffnethardconstraintaffineneural}, by relaxing the row-rank requirement.
Both settings, however, presume that the constraint functions $A$, $b_l$, and $b_u$ are given and that \eqref{eq:hardnet} has a nonempty feasible set for every input. 
For neural ODEs with Lyapunov-based constraints, we observe that neither holds \emph{a priori}: the Lyapunov and decay rate functions are themselves unknown, and one must both show that a feasible pair exists for the dynamics under consideration and learn them together with the pre-projection model $\fnomapprox$. 
In addition to being approximable, the constraint functions must also give well-behaved projection layers.
For example, we observe that instantiating \eqref{eq:hardnet} with $A(x)=\nabla V(x)^\top$, where $V$ is a Lyapunov function identified by converse theorems for exponentially stable systems \cite[Chap~4]{khalil2002nonlinear}, violates the full row-rank assumption of \cite{min_hardnet_2025} at the equilibrium, where $\nabla V$ vanishes. While this assumption was relaxed in \cite{zhao2026caffnethardconstraintaffineneural}, we observe that both methods give projection layers that rely on the pseudoinverse $\nabla V/\|\nabla V\|^2$, which diverges near the equilibrium. Thus, the projected dynamics
are not locally Lipschitz, which prevents well-posedness of the ODE.

\subsection{Contributions}
We identify the need for a class of neural ODEs that is guaranteed to be locally asymptotically stable at a given equilibrium point, whose ROA is explicitly and exactly characterized, and that is supported by universal approximation guarantees. Our contributions are stated below. 
\paragraph{Locally stable neural ODEs:} We propose a neural ODE architecture that is hard-constrained to be exponentially stable at the origin and whose ROA is exactly characterized. The dynamics $\fhat$ are designed by a projection layer that constrains an arbitrary neural network $\fnomapprox$ along the gradient field of a jointly learned Lyapunov function $\Vapprox$. 
We derive conditions under which the associated ODE is well-posed (Lemma~\ref{lem:fhat-lip}, Proposition~\ref{prop:fhat-lip}), and show that the projected dynamics solve a variant of the Zubov equation by design, so that the origin is asymptotically stable and the ROA is exactly the $1$-sublevel set of $\Vapprox$ (Proposition~\ref{prop:fhatroa}). We therefore call this the \emph{Zubov projection layer}, and derive design conditions for exponential stability (Proposition~\ref{prop:expstab}).
\paragraph{Universal approximation:} We prove that the proposed architecture is a universal approximator for any exponentially stable dynamics within its ROA. We first work with a generic projection layer, which constrains $\fnomapprox$ along another vector field $\gapprox$ by a given profile function $\psiapprox$: we give conditions where it maps into the set of continuous vector fields (Lemma~\ref{lem:fhat-cont}), and prove that its image is dense in the set of vector fields $f$ that are invariant under the projection for some approximable vector field $g$ and profile function $\psi$ (Theorem~\ref{thm:universal}). We then instantiate the result for the Zubov projection layer, where $\gapprox=\gVapprox$ and $\psiapprox=-\npdapprox(1-\Vapprox^2)$ for a positive definite function $\Vapprox$ with a unique critical point at the origin and a positive definite function $\npdapprox$  (Corollary~\ref{cor:ua-zubov}). 
Finally, we resolve the existence of an admissible target Lyapunov function $V$ and profile function $\npd$ for any sufficiently smooth dynamics with an exponentially stable equilibrium at the origin (Theorem~\ref{thm:pdesol}), and conclude that the Zubov projection layer universally approximates such dynamics on their ROAs (Corollary~\ref{cor:universal-local}).
\paragraph{Implementation:} We propose an architecture implementation that is consistent with the theory, and a bias-free learning objective that jointly promotes vector field reconstruction and ROA identification from trajectory data alone. The approach is evaluated on a couple of benchmarks and successfully compared to relevant state-of-the-art methods.



\subsection{Organization}
The remainder of the paper is organized as follows. 
Section~\ref{sec:preliminaries} reviews preliminaries on nonlinear stability, and states the considered problem. 
The proposed neural ODE architecture is introduced in Section~\ref{sec:contribution}. Its universal approximation capabilities are proven in 
Section~\ref{sec:universal-approximation}, including the existence of an admissible target Lyapunov function and profile function.
Section~\ref{sec:flow-roa} derives the well-posedness, stability, and ROA of the constrained neural ODE. A practical implementation is provided in Section~\ref{sec:implementation} with an admissible architecture and a bias-free learning objective. Section~\ref{sec:simulations} reports experimental results on nonlinear systems with nonconvex ROAs, and a conclusion is presented in Section~\ref{sec:conclusion}.

\subsection{Notation}
On $\Rn$, we use the inner product $\langle u , v \rangle\triangleq u^\top v$ with norm $\|u\|\triangleq\sqrt{u^\top u}$, and, for a matrix $M\in \R^{n\times n}$, the induced matrix norm is $\|M\|\triangleq\sup_{\|x\|=1}\|Mx\|$. 
We use the metric topology; for a set $A\subset \Rn$, the interior is denoted $A^\circ$, the closure $\overline{A}$, and the boundary $\partial A$. Open balls are denoted $B_r(x)\triangleq\{y\mid \|y-x\|<r\}$, abbreviated $B_r$ when $x=0$.

A domain $D$ is a nonempty connected open subset of $\Rn$ with $0\in D$. A function family $\mathcal{H}$ defined on the domain $D$ with range $R$ is denoted $\mathcal{H}(D,R)$. By $\mathcal{H}(\Omega,R)$ with $\Omega\subset D$ a compact subset of some domain $D$, we mean the restriction to $\Omega$ of functions in $\mathcal{H}({D,R})$. 

The $c$-sublevel set of a scalar-valued function $V$ is given by $S_c(V)\triangleq\{x\mid V(x)< c\}$. 
A scalar-valued function $\kappa:D\to R$ is \emph{positive definite} if $\kappa(0)=0$ and $\kappa(x)>0$ for all $x\in D\setminus\{0\}$.

We write $\partial_z f$ for the partial derivative of $f$ with respect to a scalar variable $z$. The gradient $\nabla V$ of a scalar-valued function $V$ is a column vector. For a general vector-valued function $f$, we let $D_x^kf(x,\cdot)$ be the $k^{\text{th}}$ derivative with respect to $x$, omitting the subscript when $f$ takes a single argument, and adopt the convention $[Df]_{ij}=\partial_{x_j} f_i$. The space of $k$-continuously differentiable functions is denoted $\ck[k]{D,R}$, with norm $ \|f\|_{\mathcal{C}^k(D)}\triangleq\max_{|\alpha|\leq k}\sup_{x\in D}\|D^\alpha f(x)\|,$
where $\alpha$ is a multi-index. We write $\ck$ for $\ck[0]$.
We denote by $\llip[k]{D, R}$ and $\lip[k]{D, R}$ the subspaces of functions whose $k^{\text{th}}$ derivative is locally and globally Lipschitz.

We denote $[n]\triangleq\{1,2,\dots,n\}$. 
We write $\Re z$ for the real part of $z\in\mathbb{C}$. For $M\in \R^{n\times n}$ we denote its spectrum $\sigma(M)$, and we say that $M$ is \emph{Hurwitz} if $\Re\lambda<0$ for all $\lambda\in\sigma(M)$.

\section{Preliminaries and Problem Formulation}\label{sec:preliminaries}
We begin this section by reviewing preliminaries on dynamical systems defined by ODEs, Lyapunov stability theory, and neural ODEs. We then give the problem formulation. 
\subsection{Preliminaries}
\subsubsection{Nonlinear Autonomous Systems}
Consider the autonomous dynamical system \eqref{eq:sys}, written in terms of its \emph{flow}
$\pi$, where $\pi(\cdot, x_0)$ denotes the solution $x(\cdot)$ with $x(0)=x_0$:
\begin{equation}\label{eq:flow}
  \begin{array}{l@{\;\;}l@{\quad}l}
    \Sigma_{f, D}: &
      \forall t \in \bigl[0, T_{f,D}(x_0)\bigr), &
      \partial_t \pi(t,x_0) = f(\pi(t,x_0)), \\
    & & \pi(0,x_0) = x_0 \in D,
  \end{array}
\end{equation} where $f\in \llip{D,\Rn}$ is a vector field defined on the domain $D\subseteq \Rn$, and the existence time $T_{f,D}$ is defined by
\begin{multline}\label{eq:viable-time}
    T_{f,D}(x)\triangleq\sup\{\tau\geq 0 \mid \pi(\cdot,x) \text{ exists uniquely on } [0,\tau] \\
    \text{and belongs to } D\}.
\end{multline}
A set $M\subseteq D$ is \emph{positively invariant} if $x_0\in M$ implies that $\pi(t, x_0)\in M$ for all $t\geq 0$.
It is well-known that whenever $f\in \llip{D,\Rn}$, the initial-value problem in \eqref{eq:flow} has a unique solution if $x_0\in M$ for a positively invariant set $M$ and trajectories are bounded \cite[Theorem~3.3]{khalil2002nonlinear}. Moreover, the flow $\pi$ inherits the smoothness of the vector field $f$ on its domain of definition \cite[Corollary~6.1]{Sideris2013}. 
\subsubsection{Equilibrium Points and Lyapunov Stability}
Suppose $f(0)=0$, then from \eqref{eq:flow} we get that $M=\{0\}$ is positively invariant. We call $x=0$ an \emph{equilibrium point}. A fundamental problem is to characterize how the flow responds to perturbations of a state at equilibrium.
   The equilibrium point is said to be \emph{stable} if for any $\varepsilon>0$ there exists $\delta>0$ such that $\|x_0\|<\delta$ implies that $\|\pi(t, x_0)\|< \varepsilon$ for all $t\geq 0$. It is \emph{asymptotically stable} if it is stable and there exists $\delta>0$ such that $\|x_0\|<\delta$ implies $\pi(t,x_0)\to 0$ as $t \to \infty$, and it is \emph{exponentially stable} if the convergence is exponential.

For an asymptotically stable origin, a central question is to identify the set of initial conditions with converging flow trajectories. This set is called the \emph{Region of Attraction (ROA)}: \begin{equation*}\roa\triangleq\Bigl\{x_0\in D\;\Bigm|\; \lim\limits_{t\to\infty}\pi(t, x_0)=0\Bigr\}.\end{equation*} 
It is well-known that 
$\roa$ is an open, connected, and invariant set, and $\partial\roa$ is formed by trajectories \cite[Lemma~8.1]{khalil2002nonlinear}. We distinguish between \emph{local asymptotic stability}, where $\roa\subseteq\Rn$, and \emph{global asymptotic stability}, where $\roa=\Rn$.

Stability certification and inner estimates of the ROA can be obtained through Lyapunov functions.
Suppose $\Sigma_{f, D}$ has an equilibrium point at the origin. If there exists a positive definite function $V\in\cd{D,\R_{\geq 0}}$ such that $\gV(x)^\top f(x)<0$ for all $x\in D\setminus\{0\}$, then it is asymptotically stable \cite[Theorem~4.1]{khalil2002nonlinear}.
We call $V$ a \emph{Lyapunov function} for $\Sigma_{f,D}$. Whenever $V$ is radially unbounded, there exists a $c>0$ such that $S_c(V) \subseteq D$ forms an inner estimate of the ROA \cite[Chap~4]{khalil2002nonlinear}.
\subsubsection{Maximal Lyapunov Functions}
A system may admit several Lyapunov functions, each presenting different inner estimates of the ROA. An exact characterization of the ROA can be obtained by a so-called \emph{maximal Lyapunov function} \cite{vannelli_maximal_1985}, which is closely related to the method of Zubov \cite{zubov1961methods}. 
The following result \cite{vannelli_maximal_1985} describes this function.
\begin{lemma}[\text{\cite[Theorems~1--2]{vannelli_maximal_1985}}]\label{thm:maximal}
   Suppose $\Sigma_{f, D}$ has an asymptotically stable equilibrium point at the origin with ROA $\roa\subseteq D$. The open set $\mathcal{A}\subseteq D$ with $0\in \mathcal{A}$ coincides with $\roa$ if there exist positive definite functions $V\in\ck[1]{\mathcal{A},\R_{\geq 0}}$ and $\phi:\mathcal{A}\to\R_{\geq 0}$ such that
    \begin{equation}\label{eq:mlpv-dis}
        \forall x\in \mathcal{A},\quad \nabla V(x)^\top f(x) =-(1-V(x)^2)\phi(x),
    \end{equation}
\begin{equation}\label{eq:mlpv-da}
    V(x)\to 1 \quad \text{as} \quad x\to\overline{\partial \mathcal{A}},
\end{equation} where $\overline{\partial\mathcal{A}}$ is the closure of $\partial\mathcal{A}$ in $\overline{\Rn}$. If $f\in\lip{\roa,\Rn}$ and $f\in \ck[1]{\mathcal{N},\Rn}$ for some neighborhood $\mathcal{N}\subset\Rn$ of the origin, then the converse statement holds.
\end{lemma}
\begin{remark}\label{rmk:scaled-V}
    By Lemma~\ref{thm:maximal}, the ROA is exactly identified by $S_1(V)$. This represents a scaled variant of the maximal Lyapunov function $V_m$ derived by Theorems~1--2 in \cite{vannelli_maximal_1985} identifying the ROA with $S_\infty(V_m)$;  the scaled conditions \eqref{eq:mlpv-dis}--\eqref{eq:mlpv-da} arise from the transformation $V\triangleq\tanh(V_m)$.
\end{remark}

\subsubsection{Neural ODE}

A neural ODE \cite{chen2018neural} is a data-driven approach to modeling systems governed by \eqref{eq:sys}. A neural network $\fhat_\theta\in \llip{D,\Rn}$ parameterized by $\theta\in \Theta\subset \R^{n_\theta}$ defines a system $\Sigma_{\fhat_\theta, D}$ with flow $\hat{\pi}_\theta$ as in \eqref{eq:flow}.
The parameterization $\theta$ is learned by matching trajectories of $\hat{\pi}_\theta$ to $\pi$ for identical starting points, that is, by minimizing the trajectory reconstruction loss:
\begin{align*}
    \min_{\theta\in\Theta} \int_D\int_0^{\bar{T}_\theta(x)}\|\pi(t, x)-\hat{\pi}_\theta(t, x)\|\,dt\,dx
    \\
    \text{s.t.}\quad \hat{\pi}_\theta \text{ is the flow of }\Sigma_{\fhat_\theta,D},
\end{align*} where $\bar{T}_\theta(x)\triangleq\min\{T_{f,D}(x),T_{\theta,D}(x)\}$ with $T_{f,D}$ and $T_{\theta,D}\triangleq T_{\fhat_\theta,D}$ as in
\eqref{eq:viable-time}.
\subsection{Problem Formulation}
Next, we state the problem considered in this paper, but first, we introduce
the following assumptions on the target system. Let $f:D\to\Rn$ be a vector field with domain $D\subseteq\Rn$.
\begin{assumption}\label{ass:all}
    The function $f$ satisfies
    \begin{asslist}
        \item\label{ass:reg} $f\in\ck[1]{D,\Rn}$, and $f\in\ck[3]{\mathcal{N},\Rn}$
              for some neighborhood $\mathcal{N}\subset D$ of the origin;
        \item\label{ass:eq} $f(0)=0$; and
        \item\label{ass:hur} the Jacobian $A\triangleq Df(0)$ is Hurwitz.
    \end{asslist}
\end{assumption}
We denote by $\Sigma_{f, D}$ the dynamical system with flow $\pi$ as defined in \eqref{eq:flow}. By Assumption~\ref{ass:all}.\ref{ass:reg}, $f\in\llip{D,\Rn}$, so $\Sigma_{f,D}$ is well-posed on positively invariant
sets with bounded trajectories \cite[Theorem~3.3]{khalil2002nonlinear}. By Assumption~\ref{ass:all}.\ref{ass:eq}, the origin is an equilibrium point. Since $f\in\ck[1](\mathcal{N},\Rn)$ (Assumption~\ref{ass:all}.\ref{ass:reg}) and its linearization is Hurwitz at the origin (Assumption~\ref{ass:all}.\ref{ass:hur}), the equilibrium point is exponentially stable \cite[Corollary~4.3]{khalil2002nonlinear}. We denote by $\roa\subseteq D$ its ROA, and assume without loss of generality that $\mathcal{N}\subset\roa$. Finally, the regularity of $f$ carries over to $\pi$ in its domain of definition
\cite[Corollary~6.1]{Sideris2013}, so
Assumption~\ref{ass:all}.\ref{ass:reg} gives
\begin{equation}\label{eq:flow-regularity}
    \pi\in\ck[1]{\R_{\geq0}\times\roa,\Rn}
        \cap\ck[3]{\R_{\geq0}\times\mathcal{N},\Rn}.
\end{equation}

The restriction to exponentially stable flows is mild: any asymptotically stable flow with dynamics in $\llip{D,\Rn}$ can be continuously transformed (by homotopy) to an exponentially stable flow on its ROA \cite{jongeneel_asymptotic_2026,grune_asymptotic_1999}. This assumption, and the strengthened regularity (Assumption~\ref{ass:all}.\ref{ass:reg}), will be helpful as we derive a maximal Lyapunov function for $\Sigma_{f,D}$. Computing this function is challenging, so, motivated by \cite{vannelli_maximal_1985}, our approach will center on a target Lyapunov function formed by integrating along $\pi$ in time. When $\pi$ converges exponentially close to the equilibrium point, we get integrability and boundedness properties useful for our analysis. 

Our problem can now be stated.\\
\begin{problem}
Assume $f$ is unknown, but we have data on $\pi$:
\begin{equation}\label{eq:data}
    \mathcal{T}=\Bigl\{\bigl(x_i,\,t_{i,k},\,\pi(t_{i,k},x_i)\bigr)
    \;\Bigm|\; i\in[N],\; k\in[K_i]\Bigr\},
\end{equation}
where $x_1,\dots,x_N\in D$ are sampled initial conditions and $0=t_{i,0}<t_{i,1}<\cdots<t_{i,K_i-1}< T_{f,D}(x_i)$
are observation times along the trajectory starting in $x_i$ for each $i\in[N]$. We want to construct a neural model $\fhat\approx f$ and a maximal Lyapunov function $\Vapprox$ for the corresponding neural ODE $\Sigma_{\fhat,D}$. This means that $S_1(\Vapprox)$ identifies the ROA of $\Sigma_{\fhat,D}$, denoted by $\roaapprox$, and can be used as an estimate of $\roa$.
\end{problem}
\section{Locally Stable Neural ODEs}\label{sec:contribution}
In this section, we present our solution to the problem.
Consider an arbitrary vector field $\fnomapprox \in\ck[]{D,\Rn}$ with $\fnomapprox(0)=0$, and a positive definite function $\Vapprox\in\ck[1]{D,\R}$ with $\Vapprox(0)=0$. We design a projection layer that maps $\fnomapprox$ to a vector field $\fhat$ for which the mentioned $\Vapprox$ solves the variant of the Zubov equation \eqref{eq:mlpv-dis}--\eqref{eq:mlpv-da} in Lemma~\ref{thm:maximal}. 
Since \eqref{eq:mlpv-dis} concerns the derivative of $\Vapprox$ in the direction of $\fhat$, it suffices to modify the component of $\fnomapprox$ that is parallel to $\gVapprox$. Define the normalization $v(\gVapprox(x))\triangleq \gVapprox(x)/\|\gVapprox(x)\|$ if $\gVapprox(x)\neq 0$ and $0$ otherwise, and let
\begin{equation}\label{eq:zubovproj}
    \fhat \triangleq \underbrace{\fnomapprox-v(\gVapprox)v(\gVapprox)^\top\fnomapprox}_{\text{orthogonal component}}
    \,\underbrace{-\,v(\gVapprox)\npdapprox(1-\Vapprox^2)}_{\text{parallel component}},
\end{equation} where $\npdapprox\in\ck[]{D,\R}$. Left-multiplying $\fhat$ in \eqref{eq:zubovproj} by $\gVapprox^\top$ gives
\begin{equation}
   \gVapprox^\top \fhat =-\|\gVapprox\| \npdapprox\bigl(1-\Vapprox^2\bigr).
\end{equation} If $\Vapprox$ has a unique critical point at the origin and $\npdapprox$ is positive definite, then $\phi\triangleq \|\gVapprox\|\npdapprox$ is positive definite and \eqref{eq:mlpv-dis}--\eqref{eq:mlpv-da} holds. Then, the origin is an asymptotically stable equilibrium point \cite[Theorem~4.1]{khalil2002nonlinear}, and Lemma~\ref{thm:maximal} holds with $\mathcal{A}\triangleq S_1(\Vapprox)$.

We call \eqref{eq:zubovproj} the \emph{Zubov projection layer}, $\zpo:\triple\to L^\infty(D,\Rn)$, where $\triple\triangleq\mathcal{F}\times\mathcal{V}\times\mathcal{P}$ with
\begin{align*}
    \mathcal{F} &\triangleq \{f\in\ck[]{D,\Rn} \mid f(0)=0\}, \\
    \mathcal{V} &\triangleq \{V\in\ck[1]{D,\R} \mid V\text{ is positive definite, and} \\
    &\qquad\qquad\ \nabla V(x)=0 \iff x=0\},\quad \text{and}\\
    \mathcal{P} &\triangleq \{\npd\in\ck[]{D,\R} \mid \npd
        \text{ is positive definite}\}.
\end{align*}
We denote by $\triplehat\subseteq\triple$ a set of function approximators, such as neural networks. Restricting the domain to $E\subseteq D$ is indicated by $\mathcal{F}_E$, $\triple_E$, and so on. 
Note that any vector field $f$ satisfying Assumption~\ref{ass:all} belongs to $\mathcal{F}$.

In what follows, we prove two main results: 1) whenever $\triplehat$ is a dense subset of $\triple$, the model class $\zpo\triplehat$ universally approximates any function $f$ satisfying Assumption~\ref{ass:all} on the ROA $\roa$ of $\Sigma_{f,D}$; and 2) for every
$\fhat\triangleq\zpo(\fnomapprox,\Vapprox,\npdapprox)$ with $(\fnomapprox,\Vapprox,\npdapprox)\in\triplehat$, the system
$\Sigma_{\fhat,D}$ has an exponentially stable equilibrium point at the origin with ROA $\roaapprox=S_1(\Vapprox)$.

We first show that universal approximation of a given $f$ can be achieved if suitable learning targets for $\Vapprox$ and $\npdapprox$ exist. We then show that these exist whenever $f$ satisfies Assumption~\ref{ass:all}. Next, we establish the exponential stability of $\Sigma_{\fhat,D}$ and its ROA characterization. Finally, we propose neural architectures that belong to $\triplehat$, and a learning objective that jointly learns $\fnomapprox$, $\Vapprox$, and $\npdapprox$ from the flow data $\mathcal{T}$ as defined in \eqref{eq:data}.

\section{Universal Approximation}\label{sec:universal-approximation}
In this section, we show that the class $\zpo\triplehat$ universally
approximates any $f$ satisfying Assumption~\ref{ass:all} on the ROA $\roa$
of the associated system $\Sigma_{f,D}$.
\subsection{Universal Approximation of a Projection Layer}
We derive universal approximation conditions for a generic projection layer $\op$ that pointwise sets the component of a vector field $\fnom$ parallel to a vector field $g$ to the profile $\psi$:
\begin{definition}
Define the \emph{projection layer}
$\op:\ck[]{D,\Rn}\times \ck[]{D,\Rn}\times\ck[]{D,\R}\to L^\infty(D,\Rn)$,
\begin{equation*}
    \op(\fnom,g,\psi)\triangleq \fnom-v(g)v(g)^\top \fnom+v(g)\psi,
\end{equation*}
where $v$ normalizes $g$ as $v(g(x))\triangleq g(x)/\|g(x)\|$ if $g(x)\neq 0$ and $0$ otherwise.
\end{definition}
Observe that $\zpo(\fnom,V,\npd)= \op\big(\fnom,\gV,-\npd(1-V^2)\big).$

We first derive conditions on the inputs $(f_0, g, \psi)$ to ensure that $\op$ maps to a function in $\ck[]{D, \Rn}$. The proof is given in Appendix~\ref{app:pf-fhat-cont}.
\begin{lemma}\label{lem:fhat-cont}
    Let the functions $\fnom ,g \in \ck[]{D,\Rn}$ and $\psi \in \ck[]{D, \R}$ satisfy $\fnom(0)=0$, $\psi(0)=0$, and
\begin{equation}\label{eq:g-zero}
    \forall x\in D,\quad g(x)=0\iff x=0.
\end{equation}
Then the function $\op(\fnom,g,\psi)$ belongs to $\ck[]{D, \Rn}$. 
\end{lemma}
We now present the universal approximation theorem for $\op$.
\begin{theorem}\label{thm:universal}
    Let the functions $f,g \in \ck[]{D,\mathbb{R}^n}$ and $\psi\in\ck[]{D,\R}$ satisfy the conditions of Lemma~\ref{lem:fhat-cont} and
    \begin{equation}\label{eq:fW-limit}
        \lim\limits_{x\to 0}\frac{\|f(x)\|}{\|g(x)\|}<\infty,
    \end{equation}
    \begin{equation}\label{eq:vtf}
        \forall x\in D,\quad g(x)^\top f(x)   = \| g(x) \| \psi(x).
    \end{equation}
Let $\Omega\subset D$ be a compact subset with $0\in\Omega^\circ$, and suppose that for any $\delta_f,\delta_g,\delta_\psi > 0$ there exist $\fnomapprox,\gapprox\in \ck[]{D,\mathbb{R}^n}$ and $\psiapprox\in\ck[]{D,\R}$ that satisfy the conditions of Lemma~\ref{lem:fhat-cont}, and
\begin{equation}\label{eq:maxdelta}
    \|f-\fnomapprox\|_{\cn{\Omega}}\leq\delta_f,\hspace{5pt} \|g-\gapprox\|_{\cn{\Omega}}\leq\delta_g,\hspace{5pt} \|\psi-\psiapprox\|_{\cn{\Omega}}\leq\delta_\psi.
\end{equation}
Then, for any $\varepsilon>0$, there exist $\delta_f, \delta_g, \delta_\psi \geq 0$ and functions $\fnomapprox, \hat{g}, \hat{\psi}$ such that $\|f- \op(\fnomapprox,\gapprox,\psiapprox)\|_{\cn{\Omega}}\leq \varepsilon.$
\end{theorem}
\begin{proof}
    Let $\fhat\triangleq\op(\fnomapprox,\gapprox,\psiapprox)$ and $v_{g}(x)\triangleq v(g(x))$. By Lemma~\ref{lem:fhat-cont}, we get that $\fhat\in\ck[]{D,\Rn}$.
    Since $\|v_{\gapprox}\|=1$ on $D\setminus\{0\}$ and $\fhat(0)=f(0)$, adding and
    subtracting $v_{\gapprox}^\top f$ and $v_g^\top f$ gives, for all $x\in D$,
\begin{equation}\label{eq:upperbound-1}
\begin{split}
    \|f(x)-\fhat(x)\| 
&\leq \ \|f(x)-\fnomapprox(x)\| \\
&\quad +| v_{\gapprox}(x)^\top (\fnomapprox(x)-f(x))| \\
&\quad +|v_{\gapprox}(x)^\top f(x)-\psiapprox(x)|\\
&\leq \|f(x)-\fnomapprox(x)\|\\
&\quad +| v_{\gapprox}(x)^\top (\fnomapprox(x)-f(x))| \\
&\quad +|(v_{\gapprox}(x)-v_{g}(x))^\top f(x)|\\ 
&\quad +|v_{g}(x)^\top f(x)-\psiapprox(x)|.
\end{split}
\end{equation} 
The conditions $\psi(0)=0$, \eqref{eq:g-zero}, and \eqref{eq:vtf} give  that $v_g^\top f=\psi$. Together with Cauchy-Schwarz inequality, \eqref{eq:upperbound-1} simplifies as
\begin{equation}\label{eq:upperbound}
    \begin{split}
        \|f(x)-\fhat(x)\| &\leq \ 2\|f(x)-\fnomapprox(x)\|\\
&\quad+\|f(x)\|\|v_{g}(x)-v_{\gapprox}(x)\| \\
&\quad +|\psi(x)-\psiapprox(x)|.
    \end{split}
\end{equation}
By assumption \eqref{eq:maxdelta}, we can choose $\fnomapprox$ and $\psiapprox$ such that
\begin{equation}\label{eq:universal:p1}
    \|f-\fnomapprox\|_{\ck[]{\Omega}}\leq \frac{\varepsilon}{6}, \quad \text{ and } \quad \|\psi-\psiapprox\|_{\ck[]{\Omega}}\leq \frac{\varepsilon}{3}.
\end{equation} It remains to bound the term $\|f\|\|v_{g}-v_{\gapprox}\|$. By assumption, it vanishes at the origin, and for all $x\in D\setminus \{0\}$,
\begin{equation}\label{eq:wnlpz}
\begin{split}    
    \|v_{g}(x)-v_{\gapprox}(x)\|&=\Big\|\frac{g(x)}{\|g(x)\|}+\frac{\gapprox(x)-\gapprox(x)}{\|g(x)\|}-\frac{\gapprox(x)}{\|\gapprox(x)\|}\Big\|
    \\
    &\leq \frac{\|g(x)-\gapprox(x)\|}{\|g(x)\|}\\&\quad+\|\gapprox(x)\|\Big|\frac{1}{\|g(x)\|}-\frac{1}{\|\gapprox(x)\|}\Big|\\
    &\leq \frac{\|g(x)-\gapprox(x)\|}{\|g(x)\|}+\frac{|\|\gapprox(x)\|-\|g(x)\||}{\|g(x)\|}\\
    &\leq \frac{2\|g(x)-\gapprox(x)\|}{\|g(x)\|}.
\end{split}
\end{equation} 
Multiplying \eqref{eq:wnlpz} with $\|f\|$, we get for all  $x\in D\setminus\{0\}$,
\begin{equation}\label{eq:fvg}
    \|f(x)\|\|v_{g}(x)-v_{\gapprox}(x)\|\leq 2\frac{\|f(x)\|}{\|g(x)\|}\|g(x)-\gapprox(x)\|.
\end{equation}
By \eqref{eq:fW-limit}, there exist $\delta>0$ and $M<\infty$ such that for all $x\in B_\delta\setminus\{0\}$,  we have $\frac{\|f(x)\|}{\|g(x)\|}<M.$
On the compact set $\Omega\setminus B_{\delta}$, $f$ is bounded by continuity and $\|g\|$ is bounded away from zero by \eqref{eq:g-zero}, so there exists $R<\infty$ such that for all $x\in\Omega\setminus B_\delta$, we have $\frac{\|f(x)\|}{\|g(x)\|}<R$.
Thus, choosing $\gapprox$ such that $\|g(x)-\gapprox(x)\|_{\ck[]{\Omega}}\leq\frac{\varepsilon}{6C}$
where $C=\max\{M, R \}$ reduces \eqref{eq:fvg} to
\begin{equation}\label{eq:universal:p2}
    \|f(x)\|\|v_{g}(x)-v_{\gapprox}(x)\|\leq\frac{\varepsilon}{3}.
\end{equation}
Using \eqref{eq:universal:p1} and \eqref{eq:universal:p2} in \eqref{eq:upperbound} gives $\|f(x)-\fhat(x)\|_{\ck[]{\Omega}}\leq\varepsilon$.
\end{proof}
\subsection{Universal Approximation of a Zubov Projection Layer}
Applying Theorem~\ref{thm:universal} with targets $f\in\mathcal{F}$, $g=\gV$ for $V\in\mathcal{V}$, and $\psi=-\npd(1-V^2)$ for $\npd\in\mathcal{P}$ gives a universal approximation result for $\zpo$:
\begin{corollary}\label{cor:ua-zubov}
    Let $(f,V,\npd)\in\triple$ satisfy
    \begin{equation}\label{eq:fgV-limit}
        \lim\limits_{x\to 0}\frac{\|f(x)\|}{\|\gV(x)\|}<\infty,\quad\text{and}
    \end{equation}
    \begin{equation}\label{eq:gVtf}
        \forall x\in D,\quad \gV(x)^\top f(x)   = \| \gV(x) \| \bigl(-\npd(x)(1-V(x)^2)\bigr).
    \end{equation}
Let $\Omega\subset D$ be a compact subset with $0\in\Omega^\circ$, and suppose that $\triplehat_\Omega$ is dense in $\triple_\Omega$.
Then, for any $\varepsilon>0$, there exist $(\fnomapprox,\Vapprox , \npdapprox)\in \triplehat_\Omega$
such that $\|f-\zpo(\fnomapprox,\Vapprox , \npdapprox)\|_{\cn{\Omega}}\leq \varepsilon.$
\end{corollary}
\begin{proof}
     The class conditions in $\triple$ and Lemma~\ref{lem:fhat-cont} gives that $\zpo$ maps every $(\fnomapprox,\Vapprox , \npdapprox)$ in $\triplehat$ to a function in $\ck[]{D,\Rn}$. The proof now follows directly from Theorem~\ref{thm:universal}.
\end{proof}

Corollary~\ref{cor:ua-zubov} establishes that $\zpo\triplehat$ universally approximates any dynamics $f\in\mathcal{F}$ for which there exist target functions $V\in \mathcal{V}$ and $\npd\in \mathcal{P}$ that satisfy conditions \eqref{eq:fgV-limit}--\eqref{eq:gVtf}. Observe that condition~\eqref{eq:gVtf} matches condition \eqref{eq:mlpv-dis} in Lemma~\ref{thm:maximal} with $\phi\triangleq\|\gV\|\npd$, which suggests that $V$ can be identified as a maximal Lyapunov function. In the next section, we show that when $f$ satisfies Assumption~\ref{ass:all}, the conditions of Corollary~\ref{cor:ua-zubov} can be met for such a function.

\subsection{Existence of a Target Lyapunov Function}
We will now show that for any function $f:D\to\Rn$ that satisfies Assumption~\ref{ass:all}, there exist $V\in\mathcal{V}$ and $\npd\in \mathcal{P}$ such that the conditions of Corollary~\ref{cor:ua-zubov} are satisfied on $\roa\subseteq D$, where $\roa$ is the ROA in the associated system $\Sigma_{f,D}$. In particular, the result will establish that $V$ is a maximal Lyapunov function. The proof is given in Appendix~\ref{app:pf-pdesol}. 
\begin{theorem}\label{thm:pdesol}
    Let $f:D\to\Rn$ satisfy Assumption~\ref{ass:all}. Then $V\in \ck[1]{\roa,\R_{\geq 0}}$ defined by
    \begin{equation}\label{eq:Vm}
            V:x\mapsto\tanh\Big(\int_0^\infty \|\pi(t,x)\|^2 dt\Big),
    \end{equation}
    where $\pi$ is the flow of $\Sigma_{f,D}$, is a maximal Lyapunov function. Additionally, $V$ satisfies:
    \begin{equation}\label{eq:zgV}
        \forall x\in \roa,\quad \gV(x) = 0 \iff x=0,
    \end{equation}
    \begin{equation}\label{eq:boundedlimit}
        \lim\limits_{x\to 0}\frac{\|f(x)\|}{\|\nabla V(x)\|}<\infty,
    \end{equation} 
    and there exists a positive definite function $\omega\in\cn{\roa, \R_{\geq 0}}$ such that $V$ solves the PDE for $x \in \roa$:    
    \begin{equation}\label{eq:vpde}
        \nabla V(x)^\top f(x)=\|\gV(x)\| \big(-\omega(x)(1-V(x)^2)\big).
    \end{equation}
\end{theorem}
\begin{proofsketch}
    We first revisit the proof of Theorem 2 in \cite{vannelli_maximal_1985} to show that
    $U:x\mapsto  \int_0^\infty \|\pi(t,x)\|^2 dt$ corresponds to a maximal Lyapunov function under the image of $\tanh$ (per Remark~\ref{rmk:scaled-V}). The proof in \cite{vannelli_maximal_1985} is derived for $U_\gamma:x\mapsto \int_0^\infty \gamma(\|\pi(t,x)\|) dt$ where $\gamma\in\ck[1](\R_{\geq 0},\R_{\geq 0})$ and its derivative $D\gamma$ are monotonically increasing, positive definite functions that yield $U_\gamma$ and $\|\nabla U_\gamma\|$ bounded. We show that this is satisfied for $\gamma:r\to r^2$ under Assumption~\ref{ass:all}. This involves showing that the map $\pi(\cdot,x)$ contracts exponentially as $t\to\infty$ for sufficiently small $\|x\|$ (Lemma~\ref{lem:flow}, Appendix~\ref{app:pf-pdesol}), and that $D_x\pi(\cdot, x)$ vanishes as $t\to \infty$ for all $x\in \roa$ (Lemma~\ref{lem:dx-roa}, Appendix~\ref{app:pf-pdesol}). The remaining steps of the proof follow identically, and we conclude that $V\triangleq \tanh(U)$ is a maximal Lyapunov function in the sense of Lemma~\ref{thm:maximal}.

    The property \eqref{eq:zgV} follows from applying Cauchy--Schwarz inequality to \eqref{eq:mlpv-dis}. The property \eqref{eq:boundedlimit} is derived by noting that $\nabla V=(1-V^2)\nabla U$ and making a Taylor expansion of $f$ and $\nabla U$ at the origin. In particular, we show that $\nabla U(0)=0$, and confirm that $D^2 U(0)$ exists by Assumption~\ref{ass:all} and Lemmas~\ref{lem:flow} and \ref{lem:dx2-bound} (Appendix~\ref{app:pf-pdesol}). We further show that 
    \begin{equation}\label{eq:inf-d2u}
        0<\inf_{\|x\|=1}\|D^2 U(0)x\|<\sup_{\|x\|=1}\|D^2 U(0)x\|<\infty,
    \end{equation}
    which renders the limit \eqref{eq:boundedlimit} finite. Finally, the PDE \eqref{eq:vpde} is derived by letting $\omega:\roa\to\R$ be defined by
    \begin{equation}\label{eq:npd-def}
    \omega:x\mapsto
    \begin{cases}
        \frac{\|x\|^2}{\|\gV(x)\|}& \text{if } \gV(x)\neq 0,\\
        0&\text{else},
    \end{cases}
    \end{equation} and showing that this is a well-defined positive definite function in $\ck[]{\roa,\R_{\geq 0}}$. With this choice of $\omega$, the PDE \eqref{eq:vpde}
    is well-defined and solved by $V$ in \eqref{eq:Vm} as for $x\in \roa,$
    \begin{equation*}
        \begin{split}
        \gV(x)^\top f(x)
        &=\frac{d}{dt}V(x)=(1-V(x)^2)(-\|\gV(x)\|\omega(x)).
        \end{split}
    \end{equation*}
\end{proofsketch} 

We observe that Theorem~\ref{thm:pdesol} identifies functions $V\in\mathcal{V}_\roa$ and $\npd\in\mathcal{P}_\roa$ satisfying the conditions in Corollary~\ref{cor:ua-zubov} on $\roa$.
\begin{corollary}\label{cor:universal-local}
    Let $f:D\to\Rn$ satisfy Assumption~\ref{ass:all}, and let $\Omega\subset\roa$ be a compact subset with $0\in \Omega^\circ$. Suppose that $\triplehat_\Omega$ is dense in $\triple_\Omega$. Then, for any $\varepsilon>0$, there exists $(\fnomapprox, \Vapprox, \npdapprox)\in \triplehat_\Omega$ such that $\|f-\zpo(\fnomapprox,\Vapprox,\npdapprox)\|_{\mathcal{C}(\Omega)}\leq \varepsilon.$
\end{corollary}
\begin{proof}
   By Assumption~\ref{ass:all} and Theorem~\ref{thm:pdesol}, there exist $V\in\mathcal{V}_\roa$ and $\omega\in\mathcal{P}_\roa$ such that \eqref{eq:fgV-limit}--\eqref{eq:gVtf} are satisfied on $\roa$. Hence, the result follows by applying Corollary~\ref{cor:ua-zubov}.
\end{proof}

\section{Flow and Region of Attraction}\label{sec:flow-roa}

In this section, we let $\fhat\triangleq\zpo(\fnomapprox,\Vapprox,\npdapprox)$
for $(\fnomapprox,\Vapprox,\npdapprox)\in\triplehat$, and analyze the flow of
the dynamical system $\Sigma_{\fhat,D}$ defined as in \eqref{eq:flow}.
First, we show that $\fhat\in \llip{D,\Rn}$, so that $\Sigma_{\fhat,D}$ has a well-defined flow on positively invariant sets with bounded trajectories \cite[Theorem~3.3]{khalil2002nonlinear}. We then show that $\Sigma_{\fhat, D}$ has an asymptotically stable equilibrium point at the origin, where the ROA $\roaapprox$ is precisely identified by $S_1(\Vapprox)$. We finally show that the origin is exponentially stable.
\subsection{Lipschitz Continuity of the Projected Vector Field}
To begin with, we derive conditions on the domain of $\op$ so it maps into $\llip{D, \Rn}$. The proof is given in Appendix~\ref{app:pf-fhat-lip}.
\begin{lemma}\label{lem:fhat-lip}
Let $\fnom ,g \in \llip{D,\mathbb{R}^n}$, $\psi \in \llip{D, \mathbb{R}}$ satisfy the conditions of Lemma~\ref{lem:fhat-cont}, and
\begin{equation}\label{eq:boundedlimit-model-po}
        \lim\limits_{x\to 0}\frac{\|\fnom(x)\|}{\|g(x)\|}<\infty\quad\text{and}\quad  
        \lim\limits_{x\to 0}\frac{|\psi(x)|}{\|g(x)\|}<\infty.
    \end{equation}
   Then $\op(\fnom,g,\psi)\in  \llip{D, \Rn}$.
\end{lemma}
Applying Lemma~\ref{lem:fhat-lip} to the domain of $\zpo$ gives:
\begin{proposition}\label{prop:fhat-lip}
Let $(\fnom,V,\npd)\in\triple$ satisfy $\fnom\in \llip{D,\Rn}$, $V\in\llip[1]{D, \R}$, $\npd \in \llip{D, \mathbb{R}}$, and
\begin{equation}\label{eq:boundedlimit-model}
        \lim\limits_{x\to 0}\frac{\|\fnom(x)\|}{\|\gV(x)\|}<\infty\quad\text{and}\quad 
        \lim\limits_{x\to 0}\frac{|\npd(x)|}{\|\gV(x)\|}<\infty.
    \end{equation}
   Then $\zpo(\fnom,V,\npd)\in  \llip{D, \Rn}$.
\end{proposition}
\begin{remark}\label{rmk:lipschitz-consistency}
The conditions in \eqref{eq:boundedlimit-model} are consistent with the identifiability result in Corollary~\ref{cor:universal-local} as follows: for the true dynamics $f$ satisfying Assumption~\ref{ass:all}, the target function $V$ identified by Theorem~\ref{thm:pdesol} satisfies $\lim\limits_{x\to 0}\|f(x)\|/\|\gV(x)\|<\infty$ by \eqref{eq:boundedlimit}; similarly, $V$ and the target function $\npd$ from \eqref{eq:npd-def} satisfy
\begin{equation*}
    \begin{split}
        \lim\limits_{x\to 0}\frac{|\npd(x)|}{\|\gV(x)\|}&=\lim\limits_{x\to 0}\frac{\|x\|^2}{\|\nabla V(x)\|^2}\\&=\lim\limits_{x\to 0}\left(\frac{\|x\|}{|1-V(x)^2|\|\nabla U(x)\|}\right)^2\\&=\lim\limits_{x\to 0}\left(\frac{\|x\|}{\|D^2U(0)x + o(\|x\|)\|}\right)^2\\&<\infty,
    \end{split}
\end{equation*} where we used that $\nabla V=(1-V^2)\nabla U$ with $V(x)\to 0$ as $x\to 0$ and $U$ satisfies $\nabla U(0)=0$ and \eqref{eq:inf-d2u}. 
\end{remark}
Note that Proposition~\ref{prop:fhat-lip} shows that $\Sigma_{\fhat, D}$ has a well-defined flow on positively invariant sets with bounded trajectories \cite[Theorem~3.3]{khalil2002nonlinear}. 

\subsection{Asymptotic Stability and ROA of the Neural ODE}
We now show that $\Sigma_{\fhat, D}$ has an asymptotically stable equilibrium point at the origin, and that $\Vapprox$ is a maximal Lyapunov function. Thus, $S_1(\Vapprox)$ is a positively invariant set that exactly identifies the ROA $\roaapprox$.
\begin{proposition}\label{prop:fhatroa}
        Let $(\fnomapprox, \Vapprox, \npdapprox)\in \triplehat$ satisfy the conditions of Proposition~\ref{prop:fhat-lip}.
        Then the system $\Sigma_{\fhat, D}$ defined by $\fhat\triangleq\zpo(\fnomapprox,\Vapprox,\npdapprox)$ has an asymptotically stable equilibrium point at the origin whose ROA is exactly identified by $S_1(\Vapprox)$. 
\end{proposition}
\begin{proof}
    By design, $\fhat(0)=0$, that is, the origin is an equilibrium point of $\Sigma_{\fhat, D}$. Since $\fhat\in\llip{D,\Rn}$ by Proposition~\ref{prop:fhat-lip}, $\Sigma_{\fhat,D}$ is well-posed on positively invariant sets with bounded trajectories \cite[Chap~3]{khalil2002nonlinear}. This includes $S_1(\Vapprox)$, since for all $x\in D$,
    \begin{equation}\label{eq:vapproxtf}
        \begin{split}
            \gVapprox(x)^\top \fhat(x) =\|\gVapprox(x)\| \big(-\npdapprox(x)(1-\Vapprox(x)^2)\big),
        \end{split}
    \end{equation} where $\npdapprox\|\gVapprox\|\in\ck[]{D,\R_{\geq 0}}$ is positive definite by design. The sufficient conditions in Lemma~\ref{thm:maximal} are satisfied with $\mathcal{A}\triangleq S_1(\Vapprox)$, and yield that $\Sigma_{\fhat,D}$ has an asymptotically stable equilibrium point at the origin with ROA $\roaapprox=S_1(\Vapprox)$. 
\end{proof}
\begin{remark}\label{rmk:effective-data-domain}
    The identifiability result in Corollary~\ref{cor:universal-local} allows for both a bounded and unbounded ROA $\roa\subseteq D$; however, the model system $\Sigma_{\fhat,D}$ will generally have a bounded ROA, $S_1(\Vapprox)$, when the layer components $(\fnomapprox,\Vapprox,\npdapprox)\in\triplehat$ are instantiated by neural networks. In practice, this is not considered a limitation: observe that the finite number of sampling points generating $\mathcal{T}$ gives an effective data domain that is contained in some compact subset $K\subset \Rn$. We are thus practically interested in a bounded domain; when the origin of the target system is globally asymptotically stable, a well-trained approximator $\fhat$ has $S_1(\Vapprox)$ containing $K$.
\end{remark}

From a broader perspective, an implication of Proposition~\ref{prop:fhatroa} and Remark~\ref{rmk:effective-data-domain} is that the exactly characterized ROA lends itself to defining a certified operating region for the learned model. This can be useful in applications where it is critical to explicitly define the region supported by the model. 


Finally, to reflect the exponential stability of the true system under Assumption~\ref{ass:all}.\ref{ass:hur}, we show that the equilibrium point is exponentially stable.
\begin{proposition}\label{prop:expstab}
    Let $(\fnomapprox, \Vapprox, \npdapprox)\in \triplehat$ satisfy the conditions of Proposition~\ref{prop:fhatroa}. If additionally, for small $\|x\|$,
        \begin{align}
            &k_1\|x\|^2\leq \Vapprox(x)\leq k_2\|x\|^2,\label{eq:vapproxquad}\\
            & \npdapprox(x)\|\gVapprox(x)\|\geq k_3\|x\|^2\label{eq:vapproxtfquad}
        \end{align} for some positive constants $\{k_i\}_{i\in [3]}$, then the origin is exponentially stable.
\end{proposition}
\begin{proof}
    Using \eqref{eq:vapproxtfquad} together with the result \eqref{eq:vapproxtf}, we get $\gVapprox(x)^\top \fhat(x)\leq -k_4\|x\|^2$ for small $\|x\|$ and some positive constant $k_4$. Then, together with \eqref{eq:vapproxquad}, Theorem~4.10 in \cite{khalil2002nonlinear} can be applied to certify that the origin is exponentially stable.
\end{proof}
\begin{remark}\label{rmk:Vexpstable}
    The conditions \eqref{eq:vapproxquad}--\eqref{eq:vapproxtfquad} are consistent with the identification result in Corollary~\ref{cor:universal-local}: from the proof of Theorem~\ref{thm:pdesol}, the learning targets $V$ and $\npd$ satisfy    
    \begin{align}
        &l_1\|x\|^2\leq V(x)\leq l_2\|x\|^2,\label{eq:Vquad}\\
        &\npd(x)\|\nabla V(x)\| = \|x\|^2\label{eq:Vtfquad},
    \end{align} for some positive constants $\{l_i\}_{i\in[2]}$. In particular, \eqref{eq:Vquad} follows from $V=\tanh(U)$ with $U(x)=\frac{1}{2}x^\top D^2U(0)x+o(\|x\|^2)$ for small $\|x\|$ with $D^2U(0)$ satisfying \eqref{eq:inf-d2u} and $\tanh(s)=s+o(|s|)$ for small $|s|$, and \eqref{eq:Vtfquad} follows by definition \eqref{eq:npd-def}.
\end{remark}

\section{Implementation}\label{sec:implementation}
In this section, we give neural network architectures for the function approximators $\triplehat$ on $D=\Rn$. We propose a bias-free learning objective targeting both the approximations $\fhat\approx f$ and $\roaapprox\approx\roa$ using flow data $\mathcal{T}$ as defined in \eqref{eq:data}.
\subsection{Model Architecture}
The below architectures for $\triplehat$ satisfy the respective class conditions. We also comment on their expressivity, and whether $\triplehat_\Omega$ is dense in $\triple_\Omega$ for compacts $\Omega\subset \Rn$ with $0\in\Omega^\circ$.

The class $\widehat{\mathcal{F}}$ can be modeled by $\fhat_\theta\triangleq h_\theta-h_\theta(0)$ where $h_\theta\in\ck[\infty]{\Rn,\Rn}$ is a smooth neural network. It can be designed using smooth activation functions such as Sigmoid Linear Unit (SiLU), defined by $x\mapsto x/(1+\exp(-x))$, which belongs to $\ck[\infty]{\R,\R}$. This ensures that $\widehat{\mathcal{F}}\subseteq\mathcal{F}$. Further, since SiLU is nonpolynomial, Theorem~1 in \cite{leshno_multilayer_1993} certifies that $\widehat{\mathcal{F}}_\Omega$ is a dense subset of $\mathcal{F}_\Omega$. This holds under translation: for $f(0)=0$,  
\begin{equation*}
    \begin{split}
             \|f-\fhat_\theta\|_{\ck[](\Omega)}&=\|(f-f(0))-(h_\theta-h_\theta(0))\|_{\ck[](\Omega)}\\
             &\leq \|f-h_\theta\|_{\ck[](\Omega)}+\|f(0)-h_\theta(0)\|_{\ck[](\Omega)}.
    \end{split}
 \end{equation*}
 
The class $\widehat{\mathcal{V}}$ can be modeled by $\Vapprox_\theta: x\mapsto \|g_\theta(x)-g_\theta(0)\|^2/2$ where $g_\theta\in \ck[1]{\Rn,\R^p}$ for $p>0$ is a bi-Lipschitz neural network \cite{wang2024monotone, cheng_learning_2024} formed with Rectified Huber Unit (ReHU) \cite{Manek2023thesis} activation functions: for some $d>0$, $\mathrm{ReHU}(x)=0$ for $x\le0$, $x^2/(2d)$ for $0<x<d$, and $x-d/2$ otherwise. 
This forms a Polyak-Łojasiewicz network (PLNet) \cite{wang2024monotone}, satisfying, for all $x\in \Rn$ and some $C>0$,
\begin{equation}\label{eq:gVapproxlb}
    \|\gVapprox_\theta(x)\|^2\geq C\Vapprox_\theta(x).
\end{equation}Since $g_\theta$ is injective, $\Vapprox_\theta$ is positive definite, and \eqref{eq:gVapproxlb} gives that the origin is the unique critical point. Thus, we conclude  $\widehat{\mathcal{V}}\subseteq\mathcal{V}$. Moreover, the bi-Lipschitz structure of $g_\theta$ renders $\Vapprox_\theta$ bounded by quadratic functions (condition \eqref{eq:vapproxquad} in Proposition~\ref{prop:expstab}).
Generally, neural networks with $\mathrm{ReHU}$ activations are universal approximators in $\ck[1]{\Omega,\R}$ by Theorem 2.1 in \cite{li_simultaneous_1996}. Whether the PLNet design renders $\widehat{\mathcal{V}}_\Omega$ dense in $\mathcal{V}_\Omega$ is open, however, it suitably models Lyapunov functions \cite{wang2024monotone, cheng_learning_2024}: sublevel sets of $\Vapprox_\theta$ are topological balls, a structural property of Lyapunov functions \cite{wilson_structure_1967}, and admit nonconvex shapes.  

The class $\widehat{\mathcal{P}}$ can be modeled by $\npdapprox_\theta:x\mapsto \mathrm{ReHU}(N_\theta(x)-N_\theta(0))+c\|x\|$, where $N_\theta\in \ck[1]{\Rn,\R}$ is a neural network with $\mathrm{ReHU}$ activations and $c>0$ is sufficiently small. Thus, $\npdapprox_\theta$ is a positive definite function in $\ck[]{\Rn,\R}$, so $\widehat{\mathcal{P}}\subseteq\mathcal{P}$. Together with \eqref{eq:gVapproxlb} and the quadratic lower bound on $\Vapprox_\theta$, the term $c\|x\|$ ensures that $\npdapprox_\theta\|\gVapprox_\theta\|$ is lower bounded by a quadratic function, ensuring condition \eqref{eq:vapproxtfquad} in Proposition~\ref{prop:expstab}. Because of the output layer, standard universal approximation theorems do not directly show that $\widehat{\mathcal{P}}_\Omega$ is dense in $ \mathcal{P}_\Omega$.

\subsection{Optimization Problem}
Suppose we have a dataset $\mathcal{T}$ consisting of sampled trajectories of the flow $\pi$ as defined in \eqref{eq:data}. Let the approximate system $\Sigma_{\fhat_\theta,D}$ be defined by $\fhat_\theta\triangleq\zpo(\fhat_{0,\theta},\Vapprox_{\theta},\npdapprox_{\theta})$ for $(\fhat_{0,\theta},\Vapprox_{\theta},\npdapprox_{\theta})\in\triplehat$. Since identifiability is only guaranteed within the ROA $\roa$ of the true system $\Sigma_{f,D}$ (Corollary~\ref{cor:universal-local}), our approach is to jointly identify the vector field $\fhat_\theta\approx f$ within the model ROA $\roaapprox_\theta$, and to match this set to $\roa$.

To learn $\fhat_\theta$ within $\roaapprox_\theta=S_1(\Vapprox_\theta)$, we define the filter
\begin{equation*}
    I_{d}(x; \Vapprox_\theta)=\begin{dcases}
        1 & \text{if } \Vapprox_\theta(x)< 1,\\
        0 & \text{else},
    \end{dcases}
\end{equation*} and form the trajectory reconstruction loss by
\begin{equation}\label{eq:LdI}
    \mathcal{L}_{d}({\theta})=\int_D\int_0^{\bar{T}_\theta(x)}\|\pi(t, x)-\hat{\pi}_\theta(t, x)\|I_{d}(x; \Vapprox_{\theta})dtdx,
\end{equation} where $\hat{\pi}_\theta$ is the flow of $\Sigma_{\fhat_\theta, D}$ and $\bar{T}_\theta(x)\triangleq\min\{T_{f,D}(x),T_{\theta,D}(x)\}$, with $T_{f,D}$ and $T_{\theta,D}\triangleq T_{\fhat_\theta,D}$ as in \eqref{eq:viable-time}.
To learn the approximation $\roaapprox_\theta\approx\roa$, we introduce two bias-free regularizers. The so-called entry regularizer $\mathcal{L}_{en}$ is formed by identifying initial conditions $x\in \roaapprox_\theta^c$ for which the flow $\pi(\cdot, x)$ enters $\roaapprox_\theta$ at some point, 
\begin{equation*}
    I_{en}(x; \Vapprox_\theta, \pi, T_{f,D})=\begin{dcases}
        1 & \text{if } \Vapprox_\theta(x)\geq 1 \text{ and } \\
          & \quad \exists t\in \bigl[0, T_{f,D}(x)\bigr):\\
          & \qquad\Vapprox_\theta(\pi(t,x))<1,\\
        0 & \text{else},
    \end{dcases}
\end{equation*}
and promoting a decrease in $\Vapprox_{\theta}$ in these points,
\begin{equation}\label{eq:Len}
    \mathcal{L}_{en}(\theta)=\int_D(\Vapprox_{\theta}(x)-1)I_{en}(x; \Vapprox_{\theta}, \pi, T_{f,D})dx.    
\end{equation}
Conversely, the so-called escape regularizer $\mathcal{L}_{es}$ is formed by identifying initial conditions $x\in \roaapprox_\theta$ for which the flow $\pi(\cdot, x)$ enters $\roaapprox^c_\theta$ at some point, 
\begin{equation*}    
I_{es}(x; \Vapprox_\theta, \pi, T_{f,D})=\begin{dcases}
        1 & \text{if } \Vapprox_\theta(x)< 1 \text{ and } \\
          & \quad \exists t\in \bigl[0, T_{f,D}(x)\bigr):\\
          &\qquad \Vapprox_\theta(\pi(t,x))\geq 1,\\
        0 & \text{else},
    \end{dcases}
\end{equation*} and promoting an increase in $\Vapprox_{\theta}$ in these points,
\begin{equation}\label{eq:Les}
    \mathcal{L}_{es}(\theta)=\int_D(1-\Vapprox_{\theta}(x))I_{es}(x; \Vapprox_{\theta}, \pi, T_{f,D})dx.
\end{equation}
In effect, $\mathcal{L}_{en}$ and $\mathcal{L}_{es}$ apply opposite pressures, correcting estimates that are too conservative and too permissive, respectively, with both vanishing once $S_1(\Vapprox_\theta)$ identifies the ROA.
\begin{remark}
Our regularization relates to the classification loss of \cite{richards_lyapunov_2018}, where sampled states are labeled by ROA inclusion and the learning procedure penalizes misclassification around the decision boundary of a Lyapunov function candidate. When labels are unavailable, the approximate procedure infers them from trajectories crossing into a sublevel set of the current Lyapunov function, similar to $I_{en}$. The approach in \cite{richards_lyapunov_2018} is conservative by design, ensuring that positive labels are necessarily true.
Our method is complementary in that we target a maximal Lyapunov function---giving the maximal ROA estimate---with universal approximation guarantees.
\end{remark}

The training procedure is warm-started by learning the pre-projection model $\fhat_{0,\theta}\in\widehat{\mathcal{F}}$, parameterized by $\theta_0\in\Theta_0\subset\R^{n_0}$:
\begin{equation}\label{eq:Ldnom}
    \begin{split}
\theta_0^\ast = \argmin\limits_{{\theta_0}\in \Theta_0}\int_D\int_0^{\bar{T}_{\theta_0}(x)}\|\pi(t, x)-\hat{\pi}_{\theta_0}(t, x)\|dtdx,
    \end{split}
\end{equation}
 where $\hat{\pi}_{\theta_0}$ is the flow of $\Sigma_{\fhat_{0,\theta}, D}$. The full model $\fhat_\theta\in\zpo\triplehat$, with parameters
$\theta \in \Theta \subset \R^{n_\theta}$ containing $\theta_0$ as a subvector, is initialized with $\theta_0 = \theta_0^\ast$ and the remaining parameters randomly drawn. It is then learned with the objective
\begin{equation}\label{eq:obj-comp}
    \begin{split}
        \min\limits_{\theta\in\Theta}\Big[\mathcal{L}_{d}( \theta) + w_{es} \mathcal{L}_{es}(\theta) + w_{en} \mathcal{L}_{en}(\theta)\Big],
    \end{split}
\end{equation}where $w_{es}>0$ and $w_{en}>0$ are tunable hyperparameters. 

\section{Numerical Simulations}\label{sec:simulations}
In this section, we validate our approach on three nonlinear benchmark systems. We begin by analyzing how well the sublevel set $S_1(\Vapprox)$, which coincides with the model ROA $\roaapprox$ by Proposition~\ref{prop:fhatroa}, recovers the true ROA $\roa$. We compare classification accuracy to a SOS polynomial Lyapunov function of degree 6\cite{helponsos} as computed in \cite{liu_physics-informed_2025}. We then analyze the reconstruction error of the learned dynamics $\fhat\approx f$ and flow $\hat{\pi}\approx \pi$ within $\roaapprox$. We compare against the globally stable neural dynamics computed in \cite{manek_learning_2020} and the passive port-Hamiltonian neural dynamics computed in \cite{cheng_learning_2024}.
\subsection{Experimental Setup}
Our method is evaluated on three two-dimensional nonlinear systems with nonconvex ROA shapes. 
The datasets are generated by sampling a set of starting points uniformly in a region containing the true ROA and integrating the dynamics to obtain a trajectory dataset $\mathcal{T}$ as defined in \eqref{eq:data}. Diverging trajectories are truncated. We evaluate the ROA estimate using classification metrics, where labels are formed by observing whether trajectories converge under the true model. We evaluate the dynamics reconstruction error by illustrating the pointwise vector field error and reporting the error of the flow trajectories. 
\subsection{Results}
\begin{figure*}[tp]
  \centering
  \begin{subfigure}[b]{0.33\linewidth}
    \includegraphics[width=\linewidth]{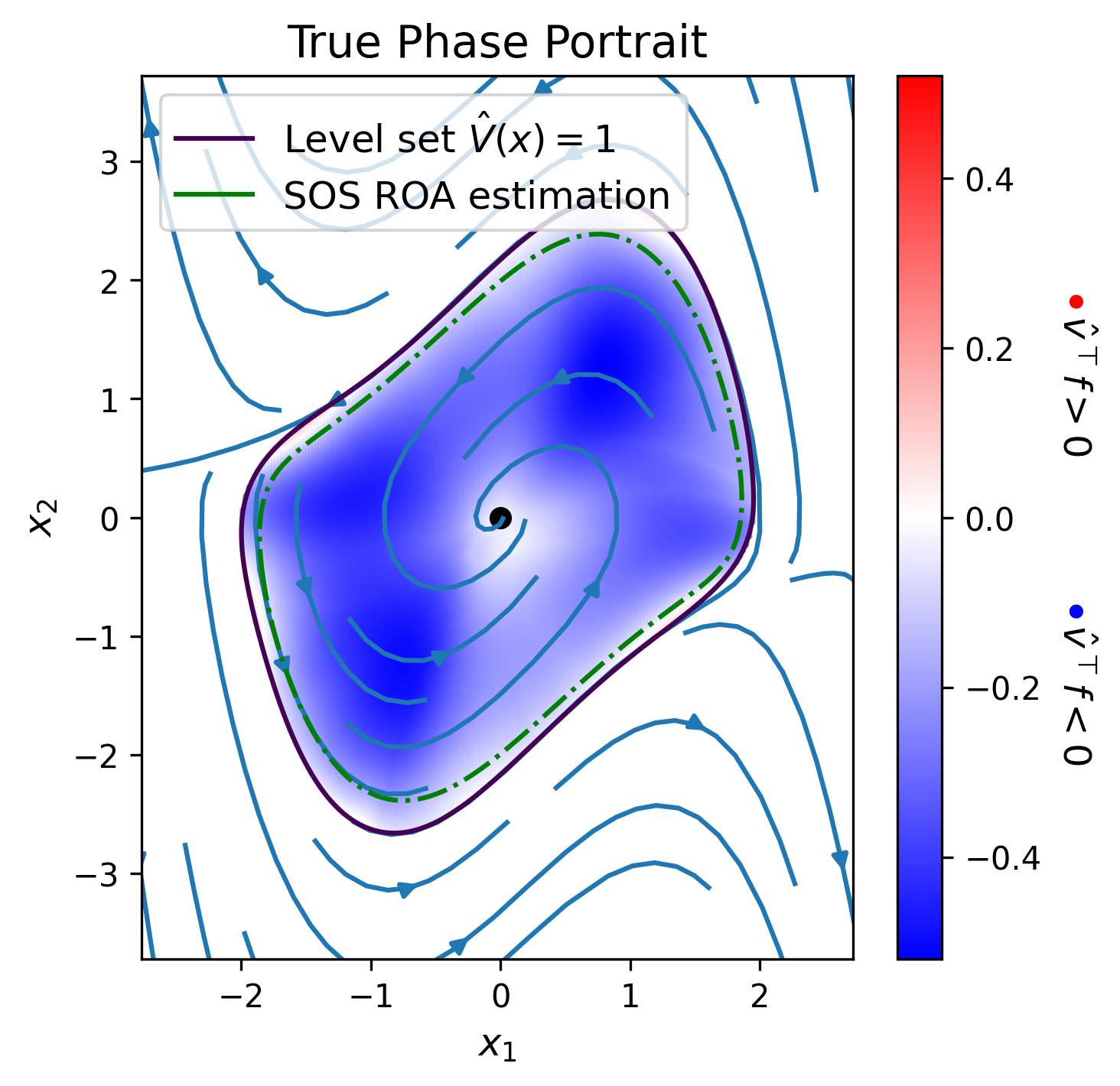}
    \caption{True flow and learned ROA}
    \label{fig:phase-portrait-VdP1-true}
  \end{subfigure}\hfill
  \begin{subfigure}[b]{0.33\linewidth}
    \includegraphics[width=\linewidth]{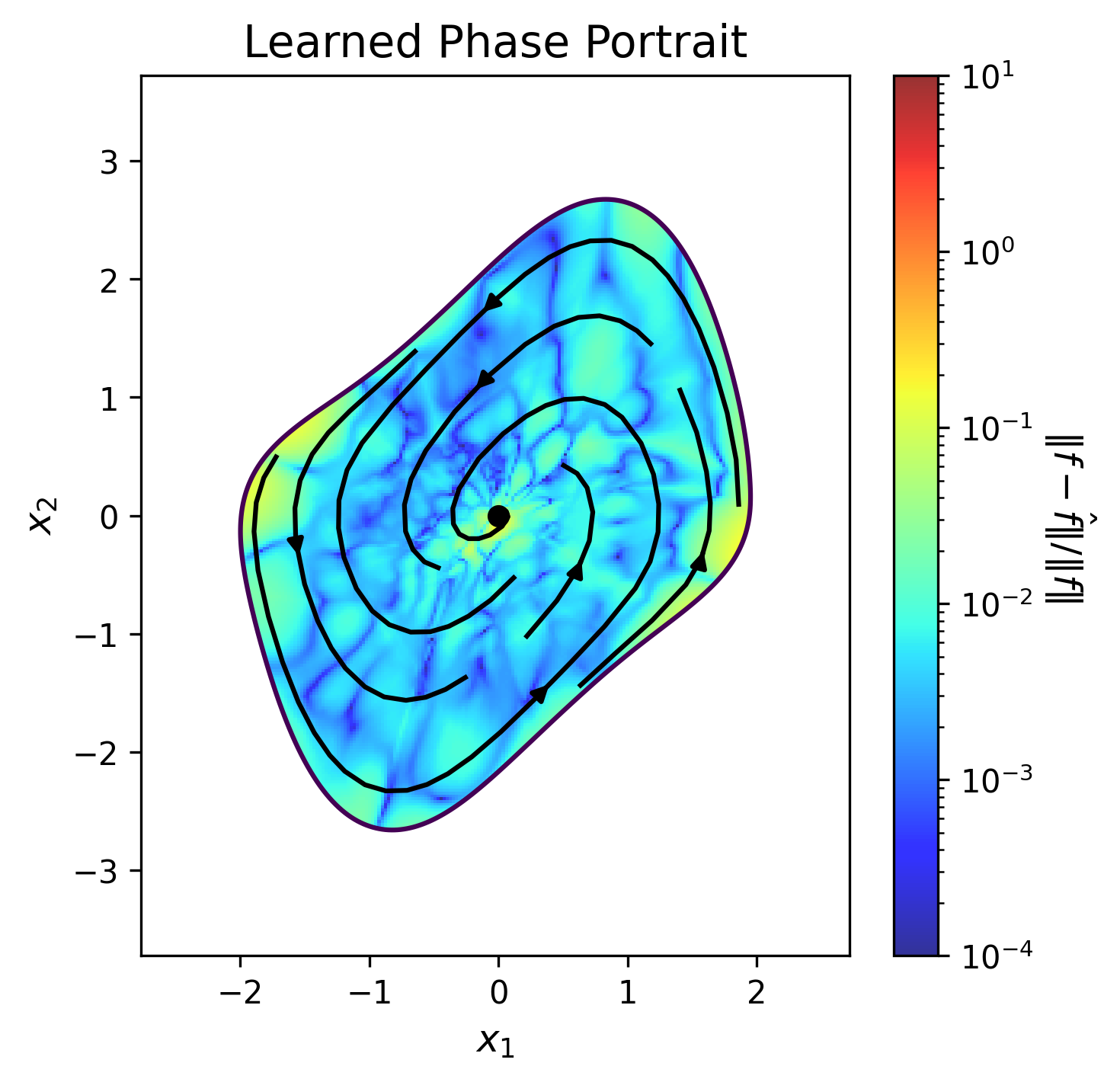}
    \caption{Learned flow -- our method}
    \label{fig:phase-portrait-VdP1-learned}
  \end{subfigure}\hfill
  \begin{subfigure}[b]{0.33\linewidth}
    \includegraphics[width=\linewidth]{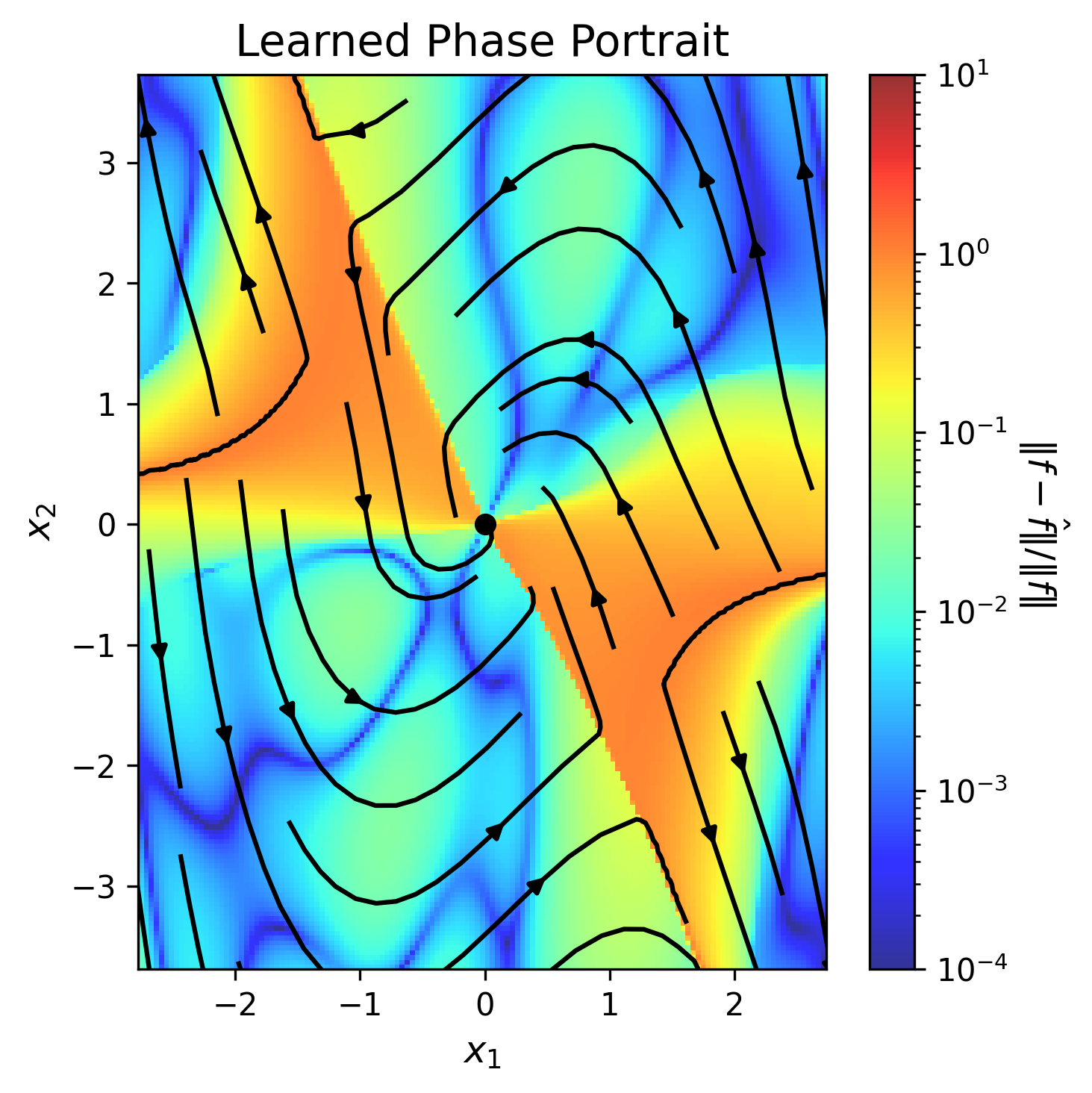}
    \caption{Learned flow -- globally stable model \cite{manek_learning_2020}}
    \label{fig:phase-portrait-VdP1-MK}
  \end{subfigure}

  \vspace{0.5em}
  {\small\itshape VdP system \eqref{eq:vdp} with $\mu=1$}
  \vspace{1em}

  \begin{subfigure}[b]{0.33\linewidth}
    \includegraphics[width=\linewidth]{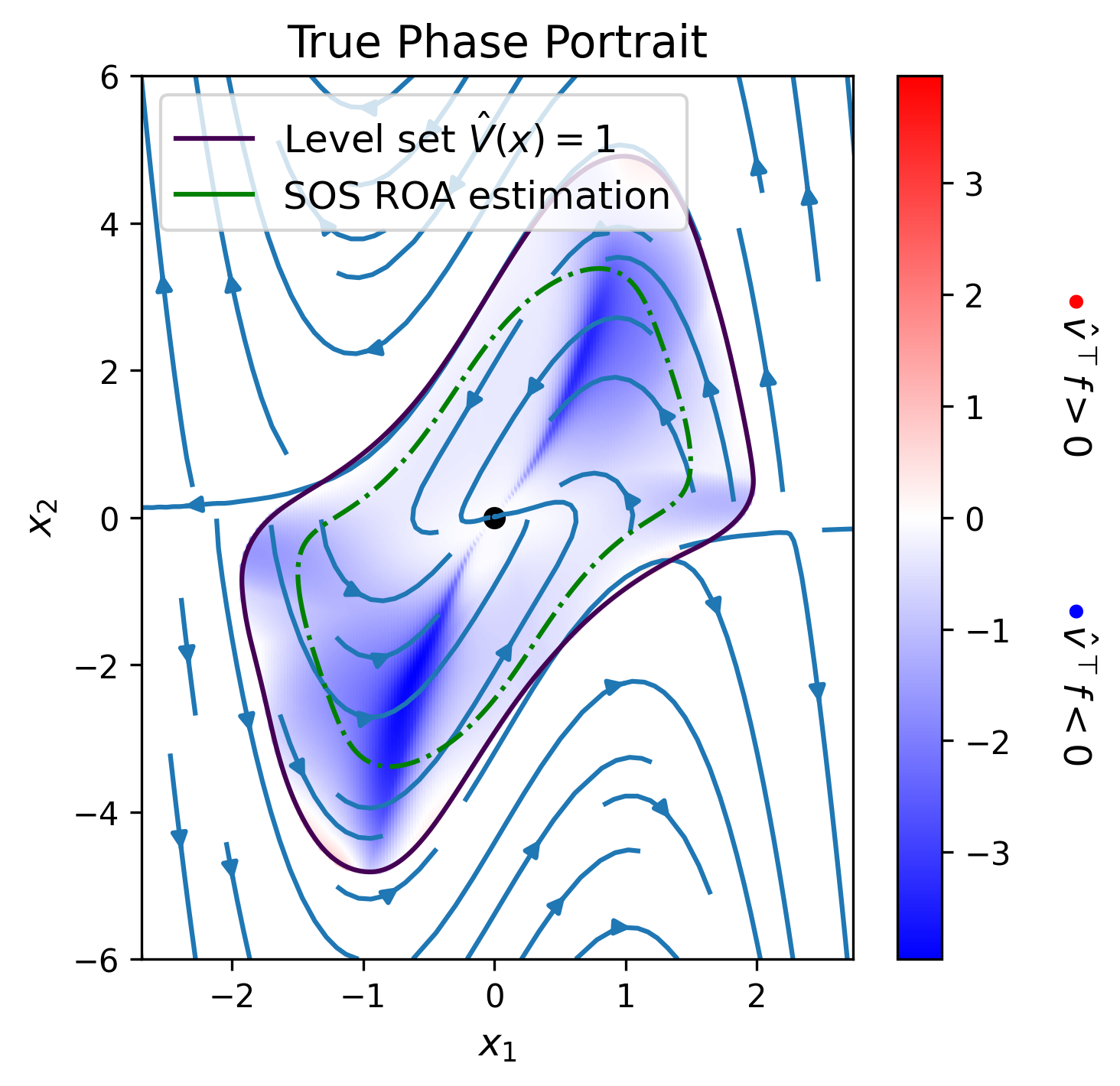}
    \caption{True flow and learned ROA}
    \label{fig:phase-portrait-VdP3-true}
  \end{subfigure}\hfill
  \begin{subfigure}[b]{0.33\linewidth}
    \includegraphics[width=\linewidth]{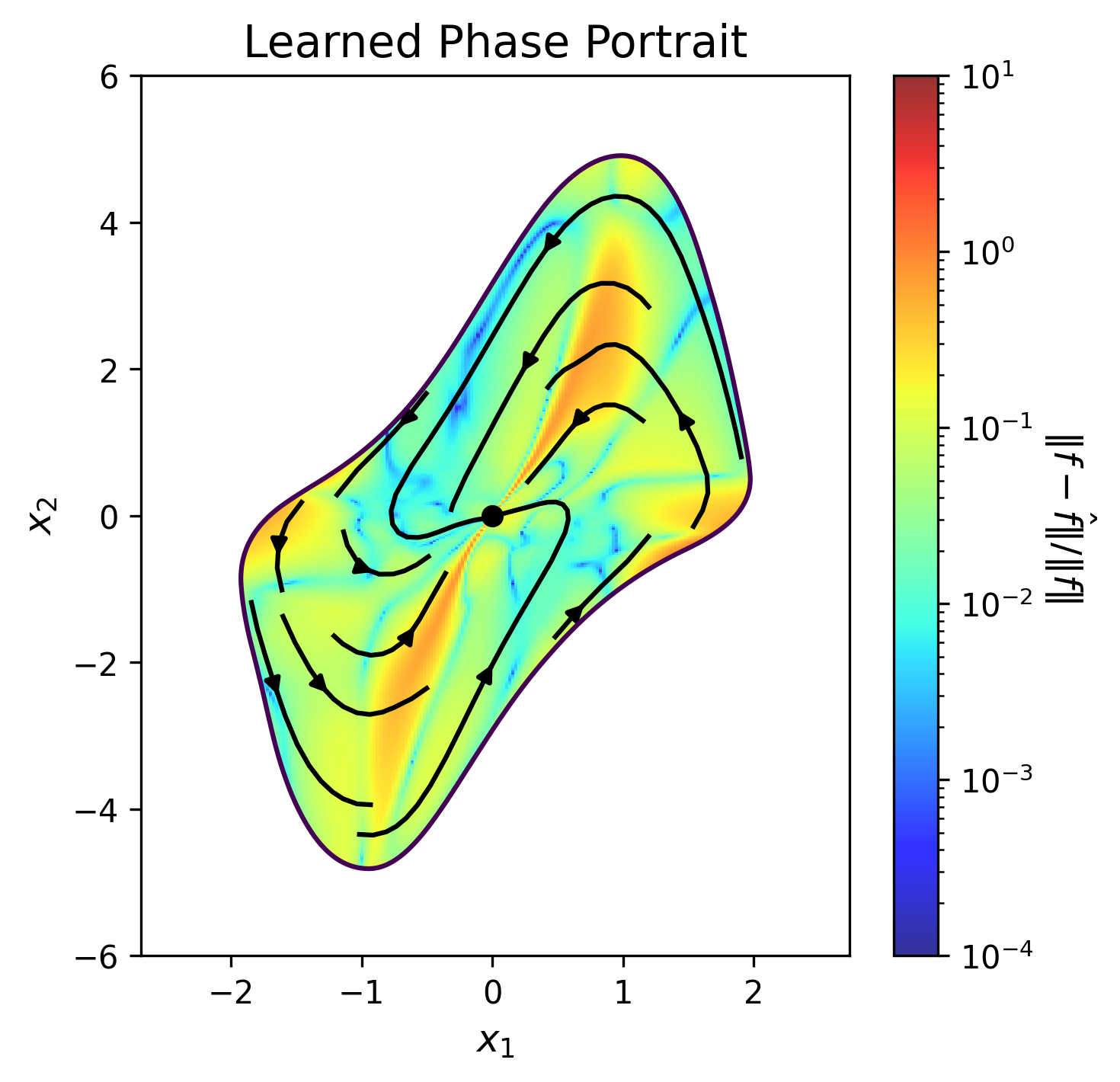}
    \caption{Learned flow -- our method}
    \label{fig:phase-portrait-VdP3-learn}
  \end{subfigure}\hfill
  \begin{subfigure}[b]{0.33\linewidth}
    \includegraphics[width=\linewidth]{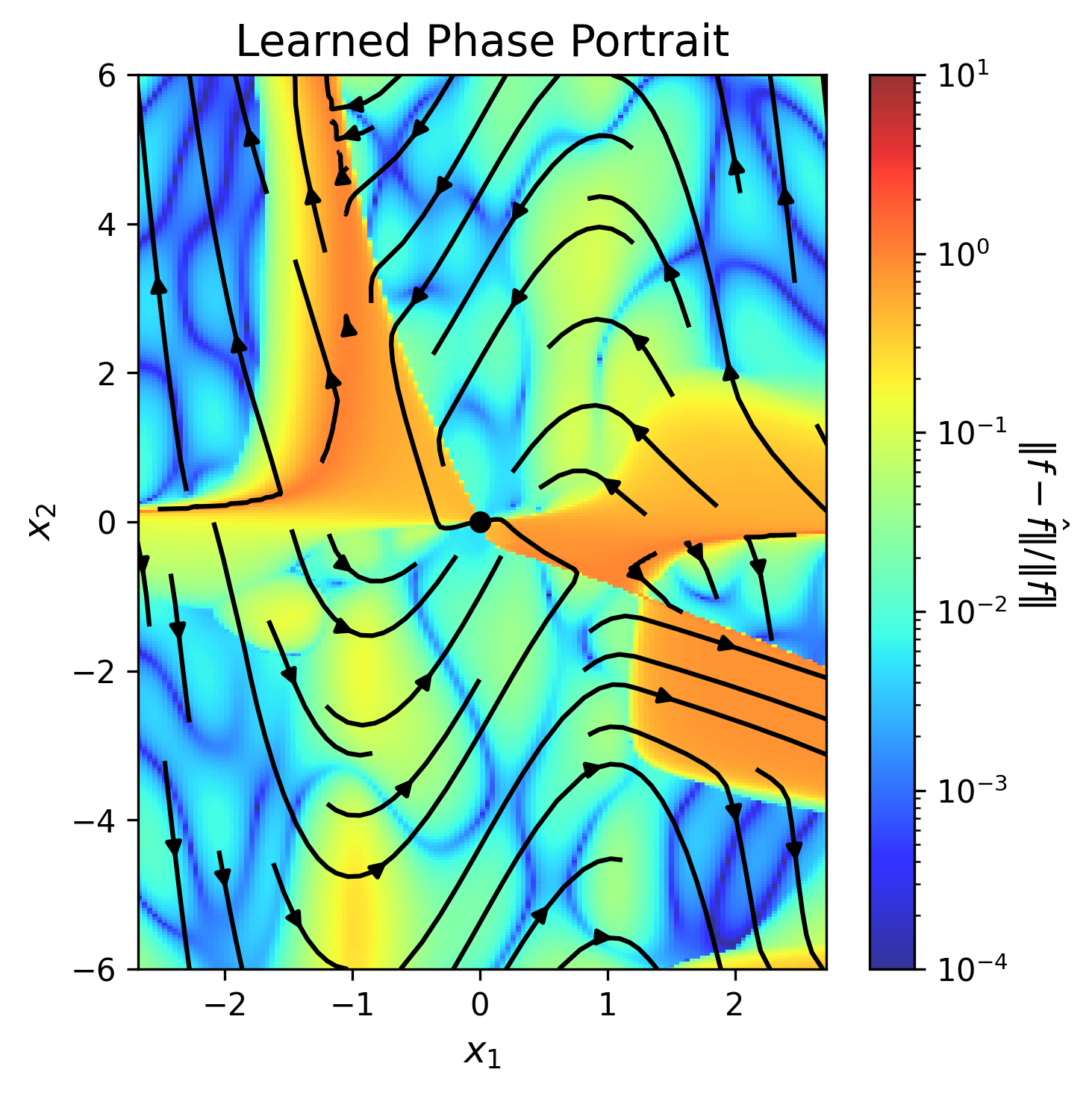}
    \caption{Learned flow -- globally stable model \cite{manek_learning_2020}}
    \label{fig:phase-portrait-VdP3-MK}
  \end{subfigure}

  \vspace{0.5em}
  {\small\itshape VdP system \eqref{eq:vdp} with $\mu=3$}
  \vspace{1em}

  \begin{subfigure}[b]{0.33\linewidth}
    \includegraphics[width=\linewidth]{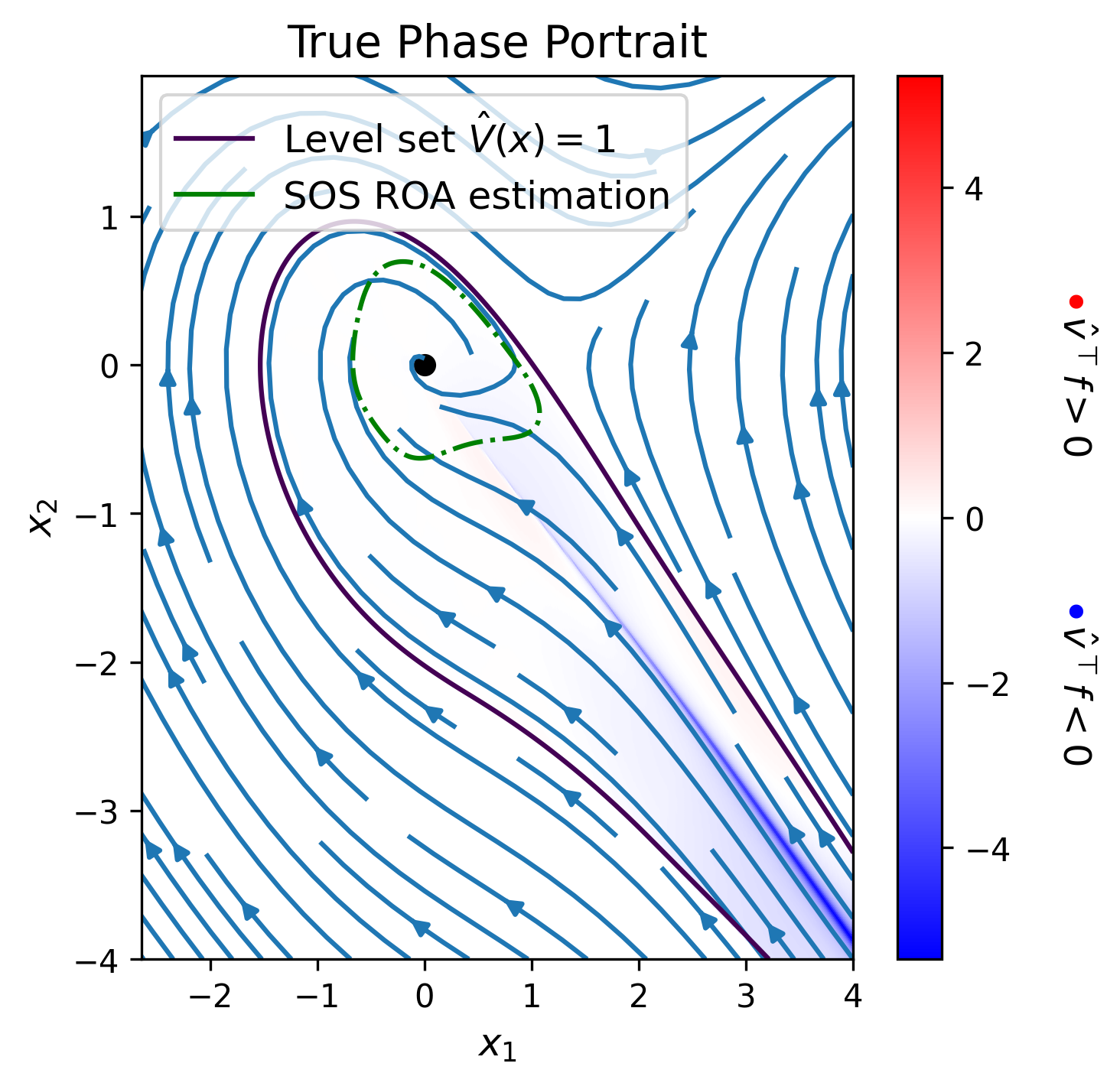}
    \caption{True flow and learned ROA}
    \label{fig:phase-portrait-TMPS-true}
  \end{subfigure}\hfill
  \begin{subfigure}[b]{0.33\linewidth}
    \includegraphics[width=\linewidth]{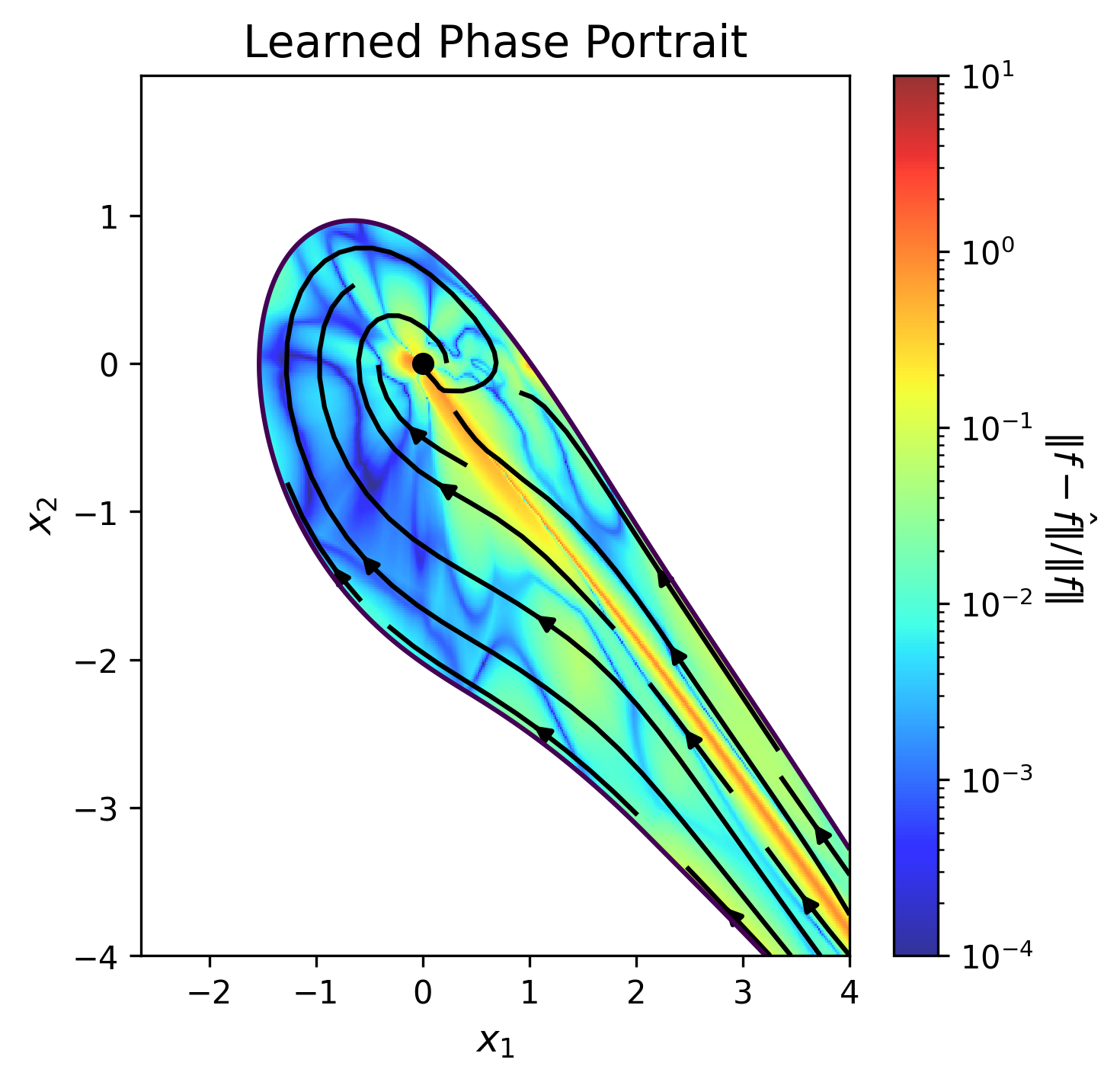}
    \caption{Learned flow -- our method}
    \label{fig:phase-portrait-TMPS-learned}
  \end{subfigure}\hfill
  \begin{subfigure}[b]{0.33\linewidth}
    \includegraphics[width=\linewidth]{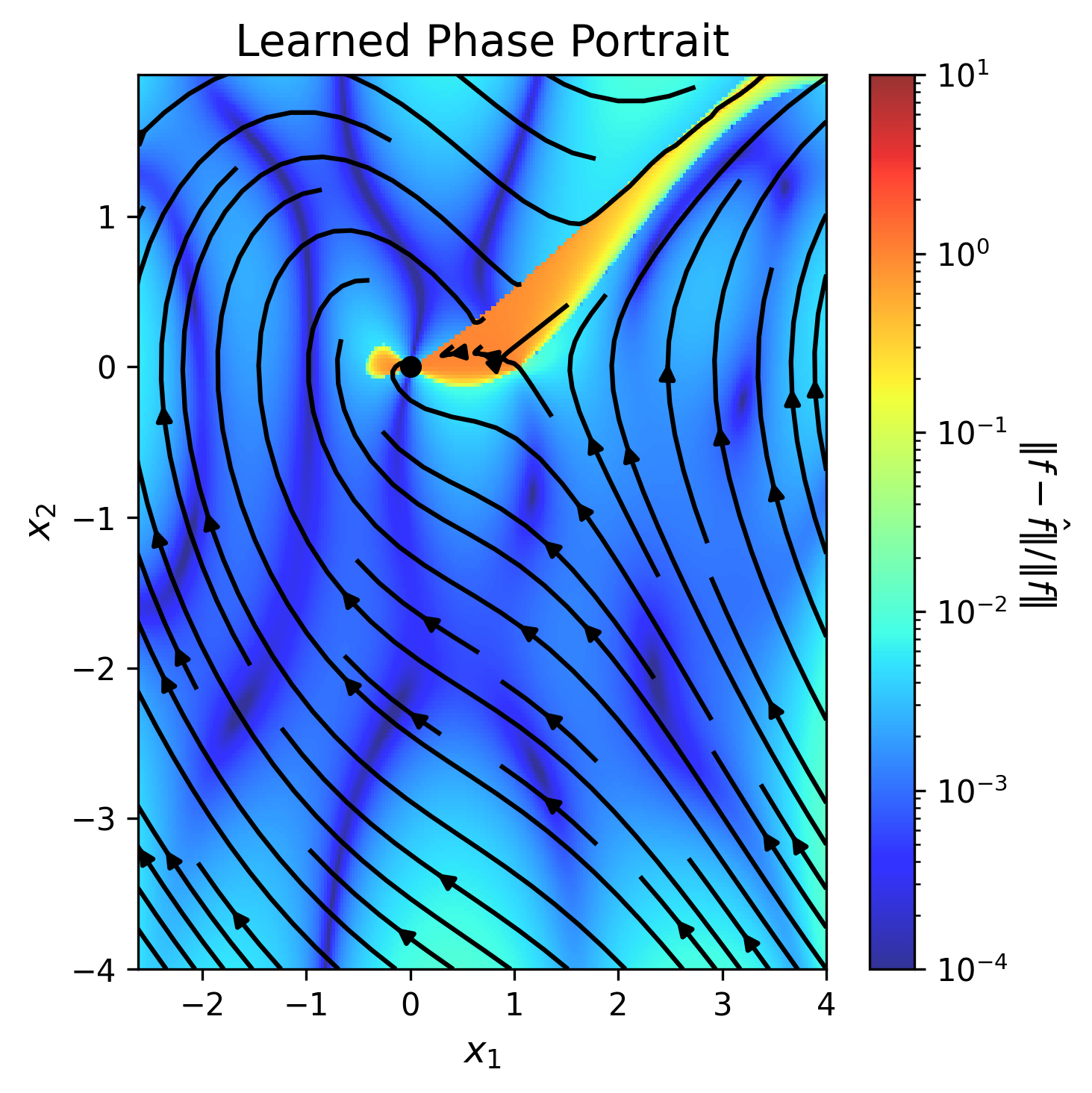}
    \caption{Learned flow -- globally stable model \cite{manek_learning_2020}}
    \label{fig:phase-portrait-TMPS-MK}
  \end{subfigure}

  \vspace{0.5em}
  {\small\itshape TMP system \eqref{eq:tmps}}

  \caption{Phase portraits for VdP (rows 1--2) and TMP (row 3). In each row, the \textbf{left} panel shows the true flow streamlines, with the contour marking the 1-level set of the learned Lyapunov function which identifies the learned ROA; the background color indicates the sign of the Lyapunov function's derivative along the true flow. The \textbf{middle} panel shows the learned flow streamlines obtained with our method, with the background color encoding the pointwise relative error of the learned vector field. The \textbf{right} panel shows the corresponding results obtained using the globally stable model from \cite{manek_learning_2020}.}
  \label{fig:phase-portrait}
\end{figure*}
Consider the time-reversed Van der Pol (VdP) system
\begin{equation}\label{eq:vdp}
    \dot{x}_1 = -x_2, \qquad \dot{x}_2 = x_1 - \mu(1 - x_1^2)x_2,
\end{equation}
with $\mu \in \{1, 3\}$, shown in the first two rows of Fig.~\ref{fig:phase-portrait}, and the two-machine power (TMP) system \cite{liu_physics-informed_2025}\footnote{Following \cite{liu_physics-informed_2025}, the SOS polynomial Lyapunov function was computed using a Taylor expansion of these nonpolynomial dynamics} 
\begin{equation}\label{eq:tmps}
    \dot{x}_1 = x_2, \qquad \dot{x}_2 = -0.5x_2 - (\sin(x_1 + \delta) - \sin(\delta)),
\end{equation}
with $\delta = \frac{\pi}{3}$, shown in the third row of Fig.~\ref{fig:phase-portrait}. In each row, the left panel shows the true flow together with the learned ROA and the ROA estimate obtained with the SOS benchmark method\cite{helponsos, liu_physics-informed_2025}, the middle panel shows the flow learned with our method, and the right panel shows the flow learned with a globally stable benchmark method \cite{manek_learning_2020}.

\begin{table}
    \centering
    \caption{ROA classification metrics (ours / SOS \cite{liu_physics-informed_2025}).}
    \label{tab:roa_classification}
    \small\setlength{\tabcolsep}{3pt}
    \begin{tabular}{lccc}
      \toprule
      & \textbf{Accuracy} & \textbf{F1}$^*$ & \textbf{Precision} \\
      \midrule
      VdP ($\mu=1$) & \textbf{0.988} / 0.933 & \textbf{0.983} / 0.902 & 1.000 / 1.000 \\
      VdP ($\mu=3$) & \textbf{0.979} / 0.848 & \textbf{0.965} / 0.673 & 0.993 / \textbf{1.000} \\
      TMP system    & \textbf{0.965} / 0.648 & \textbf{0.954} / 0.191 & 0.967 / \textbf{1.000} \\
    \bottomrule
    \end{tabular}
    \par\smallskip
    \raggedright\footnotesize $^*$F1 = harmonic mean of precision and recall:
    $\text{F1} = 2 \cdot \frac{\text{Precision} \times \text{Recall}}{\text{Precision} + \text{Recall}}$.
\end{table}

\begin{table}
    \centering
    \caption{Relative flow error (ours / \cite{manek_learning_2020} / \cite{cheng_learning_2024}).}
    \label{tab:rel_flow_error}
    \small\setlength{\tabcolsep}{3pt}
    \begin{tabular}{lcc}
      \toprule
      & \textbf{Median} & \textbf{Mean} \\
      \midrule
      VdP ($\mu=1$) & \textbf{0.006} / 0.712 / 2.148  & \textbf{0.018} / 0.603 / 2.920  \\
      VdP ($\mu=3$) & \textbf{0.078} / 0.653 / 5.712  & \textbf{0.524} / 0.571 / 22.243 \\
      TMP system    & 0.014 / \textbf{0.002} / 0.033  & 0.162 / \textbf{0.075} / 0.083  \\
      \bottomrule
    \end{tabular}
\end{table}

The left panels (Figs.~\ref{fig:phase-portrait-VdP1-true},~\ref{fig:phase-portrait-VdP3-true},~\ref{fig:phase-portrait-TMPS-true}) show that the sublevel set $S_1(\Vapprox)$ of the learned Lyapunov function $\Vapprox$ aligns with the true ROA of each system, and that $\Vapprox$ decreases along the true flow trajectories starting within this set. The proposed model architecture admits a nonconvex ROA, as is pronounced in the VdP system with $\mu = 3$ (Fig.~\ref{fig:phase-portrait-VdP3-true}), and a ROA that extends beyond the data domain, as in the TMP system (Fig.~\ref{fig:phase-portrait-TMPS-true}). Thus, $\Vapprox$ approximates well a maximal Lyapunov function for the true system, and $\roaapprox=S_1(\Vapprox)$ approximates the true ROA. As shown in Table~\ref{tab:roa_classification}, our method estimates the true ROA more accurately than the SOS method, especially for the stiffer VdP system ($\mu=3$) and the nonpolynomial TMP system. We observe some false positives (suboptimal precision scores) because our method uses only a finite number of flow trajectories to identify the ROA. By comparison, the SOS method yields an inner ROA estimate by design but requires perfect knowledge of the system equations.  


The middle panels (Figs.~\ref{fig:phase-portrait-VdP1-learned},~\ref{fig:phase-portrait-VdP3-learn},~\ref{fig:phase-portrait-TMPS-learned}) and right panels (Figs.~\ref{fig:phase-portrait-VdP1-MK},~\ref{fig:phase-portrait-VdP3-MK},~\ref{fig:phase-portrait-TMPS-MK}) compare the learned phase portraits of our method and the globally stable benchmark method \cite{manek_learning_2020}. Our method approximates the true vector field well within $\roaapprox$, yielding smooth flows that reflect the true system. The error correlates with data density: it is lower where converging trajectories accumulate and higher in areas that are rarely visited or visited only briefly. The benchmark method, in contrast, exhibits a high error in parts of the state space that the data density does not explain. Rather, the flows cut sharply across the true ROA boundary in these regions. This is explained by the globally stable model architecture, which constrains all trajectories to converge to the origin and is therefore not a universal approximator for the locally stable system, unlike our method. 

This difference is reflected in the relative flow error in Table~\ref{tab:rel_flow_error}, which on the VdP systems is lower for our method compared to both the globally stable model in \cite{manek_learning_2020} and the port-Hamiltonian model in \cite{cheng_learning_2024}. For the TMP system, our method gives a higher mean error than the benchmark methods. As seen in Fig.~\ref{fig:phase-portrait-TMPS-MK}, this system can be rendered globally asymptotically stable by only a small adjustment of the vector field, so the benchmark methods forgo little accuracy by enforcing global stability while solving the easier problem of fitting the dynamics without characterizing the ROA. The mean error notably deteriorates compared to the median for VdP with $\mu=3$ and the TMP system, so certain trajectories are relatively worse modeled than others. This suggests that there is room for improvement in the Lyapunov and profile function architectures, both of which lacking universal approximation guarantees.

\section{Conclusion and Perspectives}\label{sec:conclusion}
We proposed a class of neural ODEs that universally approximates locally exponentially stable dynamics and is, by design, exponentially stable at the origin with an exactly characterized ROA. We derived sufficient conditions under which the proposed model design---defined by a jointly learned Lyapunov function---universally approximates continuous dynamics, and then established the existence of a target Lyapunov function with the required properties. The resulting model is provably exponentially stable at the origin with an ROA exactly given by the 1-sublevel set of the jointly learned Lyapunov function. On the practical side, we proposed an implementation consistent with the theory and a learning objective based on flow-trajectory data, including bias-free regularizers that promote identification of the ROA. The approach was validated on three two-dimensional nonlinear systems with nonconvex ROA shapes.

There are multiple interesting open directions. The model class triple $\triplehat$ was assumed to be a dense subset of $\triple$; while this holds for $\widehat{\mathcal{F}}$, constructing architectures for $\widehat{\mathcal{V}}$ and $\widehat{\mathcal{P}}$ that are universally approximating remains, to the best of our knowledge, unresolved. A second direction is to learn the equilibrium point jointly with the dynamics, and to extend the framework to systems with multiple equilibria. More broadly, the same design could be generalized from equilibrium points to systems with attractors. Finally, the exactly characterized ROA can be useful in other applications where it is important to have a deterministic description of the domain supported by the model and the guarantees it provides.
\appendix
\subsection{Proof of Lemma~\ref{lem:fhat-cont}}\label{app:pf-fhat-cont}
Let $\fhat\triangleq\op(\fnom,g,\psi)$ and
$v_{g}\triangleq v(g)$. Since $v_{g}$ is continuous on
$D\setminus\{0\}$, so is $\fhat$, and it remains to analyze the behavior as $x\to0$.
Observe that for all $x\in D\setminus\{0\}$,
\begin{equation*}
    \begin{split}
        \|\fhat(x)\|&\leq \|\fnom(x)-v_{g}(x)v_{g}(x)^\top \fnom(x)\|\\
        &\quad +\|v_{g}(x)\psi(x)\|\\
        &\leq \|\fnom(x)\|+|\psi(x)|.
    \end{split}
\end{equation*} Since $\|\fnom(x)\|,|\psi(x)|\to 0$ as $x\to 0$ and $\fhat(0)=0$, we get that $\fhat(x)\to\fhat(0)$ as $x\to 0$, and conclude that $\fhat\in \ck[]{D,\Rn}$.
\subsection{Proof of Lemma~\ref{lem:fhat-lip}}\label{app:pf-fhat-lip}
To prove the Lipschitz condition locally, we first check the condition pointwise in $D\setminus \{0\}$, and then at the origin.

For $z\in D\setminus\{0\}$, let $\mathcal{N}_z\subset D$ be a neighborhood of $z$ which does not include the origin. For $x,y \in \mathcal{N}_z$, it yields:
\begin{equation}\label{eq:model-lip-all-terms}
    \begin{split}
        \|\fhat(x)-\fhat(y)\|&\leq \|\fnom(x)-\fnom(y)\|\\
        &+\|v_{g}(x)v_{g}(x)^\top \fnom(x)-v_{g}(y)v_{g}(y)^\top \fnom(y)\|\\
        &\quad+\|v_{g}(x)\psi(x)-v_{g}(y)\psi(y)\|.
    \end{split}
\end{equation} %
The first term in \eqref{eq:model-lip-all-terms} satisfies $\|\fnom(x)-\fnom(y)\|\leq L_{\fnom,z}\|x-y\|$ where $L_{\fnom, z}$ is the Lipschitz constant of $\fnom$ on $\mathcal{N}_z$. For the second term, adding and subtracting cross terms gives:
\begin{IEEEeqnarray}{rCl}
  \IEEEeqnarraymulticol{3}{l}{
    \|v_g(x)v_g(x)^\top \fnom(x)-v_g(y)v_g(y)^\top \fnom(y)\|
  } \label{eq:lip-second-term}\\
  \quad &\leq& \|\fnom(x)-\fnom(y)\| \nonumber\\
  && +\: \|v_g(x)-v_g(y)\|\big(\|\fnom(x)\|+\|\fnom(y)\|\big)\nonumber.
\end{IEEEeqnarray}
    Following the lines of \eqref{eq:wnlpz} in the proof of Theorem~\ref{thm:universal}, it can be shown that 
    \begin{equation*}
        \|v_g(x)-v_g(y)\|\leq \frac{2\|g(x)-g(y)\|}{\|g(y)\|}.
    \end{equation*} By continuity of $\fnom$ and $g$ on the compact set $\overline{\mathcal{N}}_z$, there exists a constant $R_z<\infty$ such that for all $w\in\mathcal{N}_z$, $\frac{\|\fnom(w)\|}{\|g(w)\|}< R_z$. We also have $\|g(x)-g(y)\|\leq L_{g,z}\|x-y\|$ where $L_{g,z}$ is the Lipschitz constant of $g$ on $\mathcal{N}_z$, so
    \eqref{eq:lip-second-term} reduces as:
    \begin{multline*}
      \|v_g(x)v_g(x)^\top \fnom(x) - v_g(y)v_g(y)^\top \fnom(y)\| \\
      \leq (L_{\fnom,z} + 4R_z L_{g,z})\|x-y\|.
    \end{multline*}
    The third term in \eqref{eq:model-lip-all-terms} satisfies, by adding and subtracting,
    \begin{equation}\label{eq:lip-third-term}
        \begin{split}
            \|v_g(x)\psi(x)-v_g(y)\psi(y)\|&\leq \|v_g(x)\||\psi(x)-\psi(y)|\\&\; +|\psi(y)|\|v_g(x)-v_g(y)\|.
        \end{split}
    \end{equation} 
    With the same arguments, there exists $Q_z<\infty$ such that for all $ w\in\mathcal{N}_z$, $\frac{|\psi(w)|}{\|g(w)\|}< Q_z$, so \eqref{eq:lip-third-term} reduces as:
    \begin{equation*}
        \begin{split}
             \|v_g(x)\psi(x)-v_g(y)\psi(y)\|\leq (L_{\psi,z} + 2Q_zL_{g,z})\|x-y\|.
        \end{split}
    \end{equation*}
    Using all of the above in \eqref{eq:model-lip-all-terms}, we get:
    \begin{equation}\label{eq:fhat-lip-outside}
        \begin{split}
            \|\fhat(x)-\fhat(y)\|\leq L_{\fhat,z}\|x-y\|,
        \end{split}
    \end{equation} where $L_{\fhat,z}\triangleq 2L_{\fnom,z}+(4R_z+2Q_z)L_{g,z} + L_{\psi,z} $ which certifies a local Lipschitz condition at $z\in D\setminus\{0\}$. 
    
    Next, let $z=0$ with neighborhood $\mathcal{N}_0$. First, suppose $x\in \mathcal{N}_0$ and $y=0$. With $\fhat,\fnom,\psi$ vanishing at the origin, we get:
    \begin{equation*}
        \begin{split}
           \|\fhat(x)-\fhat(0)\|
            &\leq \|\fnom(x)-v_g(x)v_g(x)^\top \fnom(x)\|\\&\quad+\|v_g(x)\||\psi(x)|\\&\leq \|\fnom(x)\|+|\psi(x)|\\&\leq L'_{\fhat,0}\|x\|,
        \end{split}
    \end{equation*} 
    where $L'_{\fhat,0}\triangleq({L_{\fnom,0}+L_{\psi,0}})$. Next, suppose $x,y\in\mathcal{N}_0\setminus \{0\}$. Then there exists a point $w\in\mathcal{N}_0\setminus\{0\}$ with neighborhood $\mathcal{N}_w$ such that $0\notin\mathcal{N}_w$, so the pair $\{x,y\}$ satisfies the condition \eqref{eq:fhat-lip-outside} for a Lipschitz constant $L_{\fhat,w}$. 
    Since $\mathcal{N}_0$ is bounded, the Lipschitz constants $\{L_{\fnom,z}\}_{z\in\mathcal{N}_0}$, $\{L_{g,z}\}_{z\in\mathcal{N}_0}$, $\{L_{\psi,z}\}_{z\in\mathcal{N}_0}$ are upper bounded \cite[Exercise 3.19]{khalil2002nonlinear}. Moreover, by condition \eqref{eq:boundedlimit-model-po}, $\{R_z\}_{z\in\mathcal{N}_0\setminus\{0\}}$ and $\{Q_z\}_{z\in\mathcal{N}_0\setminus\{0\}}$ are also upper bounded. Therefore, there exists a constant $L_{\fhat,0}''$ such that for all $w\in\mathcal{N}_0\setminus \{0\}$, $L_{\fhat,w}<L_{\fhat,0}''$. Thus, we conclude that for all pairs $x,y\in\mathcal{N}_0$,
    \begin{equation*}
        \begin{split}
            \|\fhat(x)-\fhat(y)\|\leq L_{\fhat,0}\|x-y\|,
        \end{split}
    \end{equation*} where $L_{\fhat,0}\triangleq \max\{L_{\fhat,0}',L_{\fhat,0}''\}$, which certifies a local Lipschitz condition at $z=0$.
\subsection{Proof of Theorem~\ref{thm:pdesol}}\label{app:pf-pdesol}
In this section, we derive the proof of Theorem~\ref{thm:pdesol}. 
We begin by proving three lemmas on the flow and its derivatives.

\subsubsection{Convergence of the Flow and Its Derivatives}
We begin by showing that when Assumption~\ref{ass:all} is satisfied, the flow and its derivatives converge exponentially in time with uniform rate for initial points close to the origin.
\begin{lemma}\label{lem:flow}
    Let $f:D\to\Rn$ satisfy Assumption~\ref{ass:all}. Then there exists a neighborhood $\mathcal{U}\subset \Rn$ of the origin and a constant $C_1>0$ such that for all $t\geq 0$ and $x\in \mathcal{U}$,
    \begin{align}\label{eq:flow-decay}
    \|\pi(t,x)\|\leq C_1 \exp{(-\alpha t)},
\end{align} where $0<\alpha<\min_{\lambda\in \sigma(A)}|\Re\lambda|$ with $A\triangleq Df(0)$. Moreover, 
\begin{equation}\label{eq:dx-flow-decay}
     \|D_x\pi(t,x)\|\leq C_2\exp(-\alpha t),
\end{equation} for some constant $C_2>0$,
and
\begin{equation}\label{eq:dx-at-0}
    D_x\pi(t,0)=\exp(\tilde{A}t).
\end{equation} where $\tilde{A}\in \R^{n\times n}$ similar to $A$.
\end{lemma}
\begin{proof}
    Under the conditions of Assumption~\ref{ass:all}, 
    the Hartman theorem \cite{hartman1960local,Perko2001} certifies the existence of a neighborhood $\mathcal{U}\subset\Rn$ of the origin and a $\ck[1]$-diffeomorphism $H:\mathcal{U}\to H(\mathcal{U})$ that linearizes the flow, that is for all $t\geq 0$ and $x\in\mathcal{U}$,
\begin{equation}\label{eq:HGThm}
    \begin{split}
        H(\pi(t,x))=\exp{(At)}H(x).\\
    \end{split}
\end{equation}
Since $H$ and $DH^{-1}$ are bounded by continuity, and $H(0)=0$,
\begin{equation*}
    \begin{split}
        \|\pi(t,x)\|&=\bigl\|H^{-1}\bigl(\exp{(At)}H(x)\bigr)\bigr\|\leq C_1 \|\exp(At)\|,
    \end{split}
\end{equation*} for some $C_1>0$.
We have for all $t>0$ \cite{Perko2001LinearSystems}: \begin{equation}\label{eq:matrix-exp}
    \|\exp(At)\|\leq \sum_{k=1}^d\sum_{l=0}^{l_k-1}\beta_k^{l}t^{l}\exp(-\mu_k t)
\end{equation}
where $\beta_k^l>0$ and $\mu_k=|\Re\lambda_k|$ with $\lambda_k$ being the $k^{th}$ distinct eigenvalue of $A$ having multiplicity $1\leq l_k\leq n$. For large $t$, the right-hand-side of \eqref{eq:matrix-exp} is dominated by $t^{l_{\bar{k}}-1}\exp(-\mu_{\bar{k}} t)$, where $\mu_{\bar{k}}=\min_{\lambda\in \sigma(A)}|\Re\lambda|$, which in turn is dominated by $\exp(-\alpha t)$ for any $0<\alpha<\mu_{\bar{k}}$. Hence, we can scale $\exp(-\alpha t)$ by a sufficiently large constant $C_1>0$ so that $\|\exp(At)\|\leq C_1\exp(-\alpha t)$ for all $t\geq 0$,
which implies \eqref{eq:flow-decay}.
Moreover, applying the chain rule to \eqref{eq:HGThm},
    \begin{equation*}
    \begin{split}
        D_x(H\circ\pi)(t, x) &= D_x\exp(At)H(x), \\
         DH(\pi(t,x))D_x\pi(t,x) &= \exp(At) DH(x).\\
    \end{split}
    \end{equation*} By diffeomorphic properties, $DH$ is invertible, and
\begin{equation*}
    \forall x\in\mathcal{U}, \quad \bigl(DH(x)\bigr)^{-1} = DH^{-1}(H(x)).
\end{equation*}
Thus, for all $t\geq 0$ and $x\in \mathcal{U},$ we get:
\begin{equation*}
    \begin{split}
            D_x\pi(t,x)=DH^{-1}(H(\pi(t,x)))\exp(At)DH(x).
    \end{split}
\end{equation*} 
As $DH$ and $DH^{-1}$ are bounded on $\mathcal{U}$, for some $C_2>0$,
\begin{equation*}
    \begin{split}
        \|D_x\pi(t,x)\|&=\|DH^{-1}(H(\pi(t,x)))\exp(At)DH(x)\|\\&\leq C_2 \|\exp(-\alpha t)\|,
    \end{split}
\end{equation*}which implies \eqref{eq:dx-flow-decay}.
At the origin, we get for all $t\geq 0$,
\begin{equation*}
    \begin{split}
            D_x\pi(t,0)=\bigl(DH(0)\bigr)^{-1}\exp(At)DH(0)=\exp(\tilde{A}t),
    \end{split}
\end{equation*}
where $\tilde{A}\in \R^{n\times n}$ is similar to $A$, which gives \eqref{eq:dx-at-0}.  
\end{proof}

We next show that the map $D_x\pi(\cdot,x)$ vanishes as $t\to\infty$ for all $x$ in the ROA.
\begin{lemma}\label{lem:dx-roa}
Let $f:D\to\Rn$ satisfy Assumption~\ref{ass:all}. For all $x\in\roa$, the map $D_x\pi(\cdot,x)$ satisfies  $D_x\pi(t,x)\xrightarrow[t \to \infty]{} 0.$
\end{lemma}
\begin{proof}
     By differentiating \eqref{eq:flow} with respect to $x$ and permuting the order of the partial derivatives, we get 
    the linear variational equation \cite{Sideris2013} for all $t\geq 0$ and $x\in\roa$ :
    \begin{equation}\label{eq:lve}
        \begin{split}
            \partial_t D_x\pi(t,x)=Df(\pi(t,x))D_x\pi(t,x),\quad D_x\pi(0,x)=I.
        \end{split}
    \end{equation}
    We can view this as a perturbed ODE system. For any $\xi\in\roa$, let $z_\xi:t\mapsto  D_x\pi(t,\xi)$ and rewrite \eqref{eq:lve}: for all $t\geq 0$,
    \begin{equation}\label{eq:perturbed}
        \begin{split}
                   \dot{z}_\xi(t)=Az_\xi(t)+B_\xi(t)z_\xi(t),\quad z_\xi(0)=I,
        \end{split}
    \end{equation} where $A\triangleq Df(0)$ and $B_\xi:t\mapsto Df(\pi(t,\xi))-Df(0)$.
By Assumption~\ref{ass:all}.\ref{ass:hur}, the nominal system
\begin{equation}
    \dot{z}_\xi(t)=Az_\xi(t),\quad z_\xi(0)\in\R^{n\times n},
\end{equation} has a globally exponentially stable equilibrium point at $z_\xi=0$ with a vanishing perturbation $g_\xi:(t, z_\xi)\mapsto B_\xi(t)z_\xi$. Further, since $Df \in \mathcal{C}(\roa,\R^{n\times n})$ and $\pi(t,\xi) \xrightarrow[t \to \infty]{} 0$, it yields
\begin{equation}\label{eq:btozero}
     B_\xi(t)\to 0\quad \text{as} \quad t\to\infty.
\end{equation}
 Thus, by Corollary 9.1 in \cite[Chap~9]{khalil2002nonlinear}, we conclude\footnote{Using the Lyapunov function $V=x^\top P x$ where $P\succ 0$ solves $A^\top P+PA=-I$, and $\gamma_\xi:t\mapsto \|B_\xi(t)\|$ where \eqref{eq:btozero} implies that for any $\varepsilon>0$ there exists $\eta_\xi>0$ such that $\int_{0}^t\gamma_\xi(\tau)d\tau\leq \varepsilon t+\eta_\xi$} that, for all $\xi\in \roa$, the system  \eqref{eq:perturbed} has a globally exponentially stable equilibrium point at $z_\xi=0$, so $D_x\pi(t,\xi)\xrightarrow[t \to \infty]{} 0$.
\end{proof}

To enable second-order Taylor expansion, we prove that the map $D^2_x\pi(\cdot,x)$ is bounded for small $\|x\|$ as $t\to\infty$.
\begin{lemma}\label{lem:dx2-bound}
    Let $f:D\to\Rn$ satisfy Assumption~\ref{ass:all}. The map $t,x\mapsto\|D_x^2\pi(t,x)\|$ is bounded on $\R_{\geq 0}\times \mathcal{U}$ for some neighborhood $\mathcal{U}\subset\Rn$ of the origin.
\end{lemma}
\begin{proof}
    By Assumption~\ref{ass:all}.\ref{ass:reg}, $\pi$ satisfies \eqref{eq:flow-regularity}, and Theorem~9.41 in \cite{Rudin1976pom} certifies that \eqref{eq:flow} can be differentiated twice with respect to $x$ and the order of partial derivatives permuted to give the variational equation for all $t\geq 0$ and $x\in\mathcal{N}$,
    \begin{equation}\label{eq:dx2vari}
        \begin{split}
            \partial_t D_x^2\pi(t,x)=Df(\pi(t,x))D_x^2\pi(t,x) + h(t,x),
        \end{split}
    \end{equation} where $h:(t,x)\mapsto D^2f(\pi(t,x))[D_x\pi(t,x),D_x\pi(t,x)]$.
    Similar to the proof of Lemma~\ref{lem:dx-roa}, we can view this as a perturbed system. For any $\xi\in\mathcal{N}$, let $w_\xi:t\mapsto D_x^2\pi(t,\xi)$ and rewrite \eqref{eq:dx2vari}: for all $t\geq 0$,
\begin{equation}\label{eq:w-perturb}
    \begin{split}
    \dot{w_\xi}(t)=Aw_\xi(t)+B_\xi(t)w_\xi(t)+h_\xi(t),\quad w_\xi(0)=0,
    \end{split}
\end{equation} where $A\triangleq Df(0)$, $B_\xi:t\mapsto Df(\pi(t,\xi))-Df(0)$, and $h_\xi:t\mapsto h(t,\xi)$. 
In this case, the nominal system is:
\begin{equation}\label{eq:wnom}
    \begin{split}
    \dot{w_\xi}(t)=(A+B_\xi(t))w_\xi(t),\quad w_\xi(0)\in \R^{n\times n\times n}.
    \end{split}
\end{equation} As shown in the proof of Lemma~\ref{lem:dx-roa},  the system \eqref{eq:wnom} has a globally exponentially stable equilibrium at $w_\xi=0$ for all $\xi\in\mathcal{N}$. 
Thus, by Theorem 4.14 in \cite[Chap~4]{khalil2002nonlinear}, there is a Lyapunov function $V_\xi$ for \eqref{eq:wnom} satisfying for all $t\geq 0$ and $w_\xi\in \R^{n\times n\times n}$,
\begin{align*}
    c_1^\xi\|w_\xi\|^2\leq V_\xi(t,w_\xi)&\leq c^\xi_2\|w_\xi\|^2,\\
    \partial_tV_\xi(t,w_\xi)+D_wV_\xi(t,w_\xi)^\top (A+B_\xi(t))w_\xi&\leq -c^\xi_3\|w_\xi\|^2,\\
    \|\nabla_w V_\xi(t,w_\xi)\|&\leq c^\xi_4\|w_\xi\|,
\end{align*} for some positive constants $\{c^i_\xi\}_{i\in[4]}$. 
Moreover, the perturbation $h_\xi$ is uniformly bounded for small $\|\xi\|$. Indeed, since $D^2f$ is bounded on $\mathcal{N}$ by continuity (Assumption~\ref{ass:all}.\ref{ass:reg}) and $D_x\pi$ is bounded on $\R_{\geq 0}\times \mathcal{U}$ for some neighborhood $\mathcal{U}\subset \mathcal{N}$ of the origin (Lemma~\ref{lem:flow}), there exists a constant $\delta>0$ such that for all $\xi\in\mathcal{U}$ and $t\geq 0$ we have $\|h_\xi(t)\|\leq \delta$. 
Therefore, by Lemma 9.2 in \cite[Chap~9]{khalil2002nonlinear}, we conclude that the solution $w_\xi$ of \eqref{eq:w-perturb} is globally ultimately bounded, that is, there exists a constant $b_\xi>0$ such that $\|w_\xi(t)\|\leq b_\xi$ for all $t\geq 0$. Specifically, $b_\xi=c_\xi\delta$ where $c_\xi= \frac{c^\xi_4}{c^\xi_3}\sqrt\frac{c^\xi_2}{c^\xi_1}$. Since $B_\xi$ varies continuously with $\xi$, the Lyapunov function $V_\xi$ can be designed such that the constants $\{c^i_\xi\}_{i\in[4]}$ vary continuously with $\xi$ sufficiently close to the origin\footnote{See the proof of Theorem 4.14 in \cite[Chap~4]{khalil2002nonlinear}.}. Without loss of generality, we assume this holds in $\mathcal{U}$. Then there exists a constant $c>0$ such that for all $t\geq 0$ and $\xi\in\mathcal{U}$,
\begin{equation*}
    \|D^2_x\pi(t,\xi)\|=\|w_\xi(t)\|\leq c.
\end{equation*}
\end{proof} 

\subsubsection{Proof of Theorem~\ref{thm:pdesol}}
With Lemmas~\ref{lem:flow}--\ref{lem:dx2-bound}, we are now ready to prove Theorem~\ref{thm:pdesol}. 

\begin{proofof}{Theorem~\ref{thm:pdesol}}
We revisit the proof of Theorem 2 in \cite{vannelli_maximal_1985} to show that $U:x\mapsto \int_{0}^\infty \|\pi(t,x)\|^2 dt$ corresponds to a maximal Lyapunov function under the image of $\tanh$ (per Remark~\ref{rmk:scaled-V}). The proof \cite{vannelli_maximal_1985} is derived for $U_\gamma:x\mapsto \int_0^\infty \gamma(\|\pi(t,x)\|) dt$ where $\gamma$ belongs to $\ck[1](\R_{\geq 0},\R_{\geq 0})$ and both $\gamma$ and $D\gamma$ are monotonically increasing, positive definite functions that yield $U_\gamma$ and $\|\nabla U_\gamma\|$ bounded per equations (19) and (26), respectively, in \cite{vannelli_maximal_1985}. For $\gamma:r\mapsto r^2$, observe that the regularity, monotonicity, and positive definiteness conditions are satisfied. It remains to prove that $U$ and $\|\nabla U\|$ are bounded under our hypotheses. 

By Assumption~\ref{ass:all}, Lemma~\ref{lem:flow} certifies that there exists a neighborhood $\mathcal{U}\subset\Rn$ of the origin such that for all $x\in \mathcal{U}$, the map $\|\pi(\cdot, x)\|$ is dominated by $t\mapsto C_1\exp(-\alpha t)$ where $C_1>0$ and $0<\alpha<\min_{\lambda\in\sigma(A)}|\Re\lambda|$. Let \[T:x\mapsto \inf\Bigl\{t\geq 0\;\Bigm|\;\forall \tau\geq t,\quad \pi(\tau,x)\in \mathcal{U}\Bigl\},\] with $T(x)<\infty$ for all $x\in\roa$. We get:
\begin{equation*}
    \begin{split}
            U(x)&=\int_0^\infty \|\pi(t,x)\|^2dt\\
            &=\int_0^{T(x)}\|\pi(t,x)\|^2dt +\int_{T(x)}^\infty\|\pi(t,x)\|^2dt\\
            &\leq \int_0^{T(x)}\|\pi(t,x)\|^2dt +C_1^2\int_{T(x)}^\infty \exp(-2\alpha t)dt\\&<\infty.
    \end{split}
\end{equation*}Furthermore, by Assumption~\ref{ass:all} and Lemma~\ref{lem:dx-roa}, we obtain
\begin{equation}\label{eq:gV-decay}
\begin{split}
        \|\nabla U(x)\|&\leq \int_0^\infty \|D_x (\|\pi\|^2)(t,x)\|dt\\&\leq\int_0^{T(x)} 2\|\pi(t,x)\|\|D_x\pi(t,x)\|dt\\&\quad +2C_1\int_{T(x)}^\infty \exp(-\alpha t)\|D_x\pi(t,x)\|dt\\&<\infty.
\end{split}
\end{equation}

The remaining steps of the proof of Theorem~2 in \cite{vannelli_maximal_1985} follow identically. Thus, we conclude that $V\triangleq \tanh(U)$ is a maximal Lyapunov function.  

We now turn to the properties \eqref{eq:zgV}--\eqref{eq:vpde}. First, observe that property \eqref{eq:mlpv-dis} in Lemma~\ref{thm:maximal} implies that $\|\gV\|$ is strictly positive for all $x\in\roa\setminus\{0\}$: by Cauchy--Schwarz inequality,
\begin{equation*}
    \begin{split}
         0<| \gV(x)^\top f(x)|\leq \|\gV(x)\|\|f(x)\|.
    \end{split}
\end{equation*} Hence, $\gV$ vanishes on $\roa$ only at the origin, proving \eqref{eq:zgV}.

Further, to evaluate $\lim\limits_{x\to 0}\frac{\|f(x)\|}{\|\nabla V(x)\|}$ in \eqref{eq:boundedlimit}, we make a Taylor expansion of the numerator and denominator at the origin. By Assumptions~\ref{ass:all}.\ref{ass:reg}--\ref{ass:all}.\ref{ass:eq} and Theorem 5.1 in \cite{Coleman2012}, $f$ satisfies
\begin{equation*}
    f(x) = Df(0)x+o(\|x\|)
\end{equation*} for small $\|x\|$. 
Furthermore, $\gV$ can be expanded as 
\begin{equation}\label{eq:nablaV}
    \nabla V=\nabla (\tanh(U))=(1-V^2)\nabla U.
\end{equation} Since $V(x)\to 0$ as $x\to 0$, the convergence rate of \eqref{eq:nablaV} depends on $\nabla U$. Thus, we do a first-order Taylor expansion of $\nabla U$ at the origin. 
To show that $\nabla U$ is differentiable here, it is sufficient that for small $\|x\|$ the map $D_x(\|\pi\|^2)(\cdot, x)$ is integrable and that $D_x^2(\|\pi\|^2)(\cdot,x)$ exists for $t\geq 0$ and that this map is dominated by some integrable function $\Theta:\R_{\geq 0}\to \R_{\geq 0}$ uniformly in $x$ \cite{alma9934968102456}.
We have shown that $D_x(\|\pi\|^2)(\cdot, x)$ is integrable for all $x\in \roa$ in \eqref{eq:gV-decay}. 
Further, we have that $D_x^2 (\|\pi\|^2)\in \ck[]{\R_{\geq 0}\times \mathcal{N}, \R^{n\times n}}$ from \eqref{eq:flow-regularity}, and for all $t\geq 0$ and $x\in\mathcal{N}$,
\begin{equation}\label{eq:d2xpi}
    \begin{split}
         D_x^2 (\|\pi\|^2)(t,x)
        &=2D_x\pi(t,x)^\top D_x\pi(t,x)\\
        &\quad+2\pi(t,x)^\top D^2_x\pi(t,x).
    \end{split}
\end{equation}
Note that both terms in \eqref{eq:d2xpi} are exponentially contracting under Assumption~\ref{ass:all}: Lemma~\ref{lem:flow} certifies that $ \|\pi(\cdot,x)\|$ and $\|D_x\pi(\cdot,x)\|$ contract exponentially with rate $\alpha$ as $t\to\infty$ for small $\|x\|$, and Lemma~\ref{lem:dx2-bound} certifies that $\|D_x^2\pi(\cdot,x)\|$ is uniformly bounded for small $\|x\|$. Hence, there exists a constant $C_3>0$ and a neighborhood $\mathcal{U}\subset\mathcal{N}$ of the origin where the map $ D_x^2 (\|\pi\|^2)(\cdot,x)$ is dominated by the integrable function $\Theta:t\mapsto C_3\exp(-\alpha t)$ uniformly in $x$: for all $t\geq 0$ and $x\in \mathcal{U}$,
\begin{equation*}
     \|D^2_x (\|\pi\|^2)(t,x)\|\leq \Theta(t).
\end{equation*} Thus, Theorem 2.27 in \cite{alma9934968102456} certifies that $D^2U(0)$ exists.
Hence, by Theorem 5.1 in \cite{Coleman2012} the Taylor expansion
\begin{equation*}
    \nabla U(x) = \nabla U(0)+D^2 U(0)x + o(\|x\|)
\end{equation*} holds for sufficiently small $\|x\|$.
The first term vanishes:
\begin{equation*}
    \nabla U(0)=\int\limits_0^\infty 2\pi(t,0)^\top D_x\pi(t,0)dt=0,
\end{equation*} since $\pi(t,0)=0$ for all $t\geq 0$ by property of the equilibrium point. From \eqref{eq:d2xpi} and Lemma~\ref{lem:flow}, the second term computes as
\begin{equation*}\label{eq:d2u0}
    \begin{split}
        D^2U(0)&=\int\limits_0^\infty  D_x^2 (\|\pi\|^2)(t,0)dt\\
        &=\int\limits_0^\infty 2D_x\pi(t,0)^\top D_x\pi(t,0)+2\pi(t,0)^\top D^2_x\pi(t,0)dt\\&=\int\limits_0^\infty2\exp{(\tilde{A}^\top t)}\exp{(\tilde {A} t)}dt,
    \end{split}
\end{equation*}
where $\tilde{A}\in \R^{n\times n}$ is similar to $A$. Since $\tilde{A}$ is Hurwitz, the right-hand-side in the previous equation converges, yielding
\begin{equation*}
    0<\inf_{\|x\|=1}\|D^2 U(0)x\|<\sup_{\|x\|=1}\|D^2 U(0)x\|<\infty.
\end{equation*} 
This is sufficient to conclude that the limit is bounded: 
 \begin{equation*}
 \begin{split}
     \lim\limits_{x\to 0}\frac{\|f(x)\|}{\|\nabla V(x)\|}&=\lim\limits_{x\to 0}\frac{\|f(x)\|}{|1-V(x)^2|\|\nabla U(x)\|}\\
     &= \lim\limits_{x\to 0}\frac{\|f(x)\|}{\|\nabla U(x)\|}\\
     &= \lim\limits_{x\to 0}\frac{\|Df(0)x+o(\|x\|)\|}{\|D^2U(0)x + o(\|x\|)\|}<\infty.
 \end{split}
\end{equation*}

Finally, consider the function $\omega:\roa\to\R$ defined by
    \begin{equation}\label{eq:npd-def-pf}
    \omega:x\mapsto
    \begin{cases}
        \frac{\|x\|^2}{\|\gV(x)\|}& \text{if } \gV(x)\neq 0,\\
        0&\text{else}.
    \end{cases}
\end{equation} Since both $\|\cdot\|$ and $\|\nabla V\|$ are positive definite on $\roa$, so is $\npd$. Moreover, $\|\cdot\|,\|\nabla V\|\in\ck[](\roa,\R_{\geq 0})$, so to show that $\omega\in \ck[](\roa,\R_{\geq 0})$ it suffices to show that it is not singular at the origin. To this end, observe that for all $x\in\roa\setminus\{0\},$
\begin{equation*}
    \begin{split}
       \|x\|^2=|\nabla U(x)^\top f(x)|
        \leq \|\nabla U(x)\|\|f(x)\|,       
    \end{split}
\end{equation*} 
which implies that
\begin{equation*}
\begin{split}
    \omega(x)=\frac{\|x\|^2}{\|\nabla V(x)\|}
                 \leq \frac{\|\nabla U(x)\|\|f(x)\|}{|1-V(x)^2|\|\nabla U(x)\|}=\frac{\|f(x)\|}{|1-V(x)^2|}.
\end{split}
\end{equation*}
    Since $f(x)\to 0$ and $V(x)\to 0$ as $x\to 0$, this implies that 
    $\omega(x)\to 0$ as $x\to 0$. Therefore, we conclude that $\omega\in\ck[]{\roa, \R_{\geq 0}}$. With $\omega$ as in \eqref{eq:npd-def-pf}, the PDE for $x\in\roa$
    \begin{equation*}
        \gV(x)^\top f(x)=(1-V(x)^2)\left(-\|\gV(x)\|\omega(x)\right)
    \end{equation*}
    is well-defined and solved by $V$ as
    \begin{equation*}
        \begin{split}
            \gV(x)^\top f(x)&=\frac{d}{dt}V(x)\\
    &=\frac{d}{dt}\tanh\Big(\int_0^\infty \|\pi(t,x)\|^2 dt\Big)\\
        &= (1-V(x)^2)\frac{d}{dt}\int_0^\infty\|\pi(t,x)\|^2dt\\
        &= (1-V(x)^2)\\&\quad\cdot \frac{d}{dt}\int_0^\infty\|\nabla V(\pi(t, x))\|\omega(\pi(t, x))dt\\
        &=(1-V(x)^2)(-\|\gV(x)\|\omega(x)).
        \end{split}
    \end{equation*}
\end{proofof}

\section*{Acknowledgment}
The authors thank Mikael Johansson for important comments on the manuscript draft.
\paragraph{Declaration of AI use:} Part of the computer code was written with the assistance of Claude Opus 5.5 (Anthropic). All AI-generated code was reviewed and tested by the authors. The authors take full responsibility for the manuscript and the codebase. 
\section*{References}
\bibliographystyle{IEEEtran} 
\bibliography{biblio-formatted}

--------------------------------------------


\end{document}